\documentclass[12pt]{amsart}
\usepackage{txfonts}
\usepackage{amsfonts}
\usepackage{amsfonts}
\usepackage{amssymb,latexsym}
\usepackage{cite}
\usepackage{enumerate}
\newtheorem{theorem}{Theorem}

\newtheorem{remark}[theorem]{Remark}

\newtheorem{lemma}[theorem]{Lemma}
\newtheorem{definition}[theorem]{Definition}

\newcommand {\p}{\partial}

\numberwithin{equation}{section}
\numberwithin{theorem}{section}

\usepackage{titletoc}

\usepackage{hyperref}
\hypersetup{hypertex=true,colorlinks=true,linkcolor=blue,anchorcolor=blue,citecolor=blue}
\begin{document}
\title[The homogeneous k-Hessian equation]{The exterior Dirichlet problem for homogeneous $k$-Hessian equation}
\author{Xi-Nan Ma} 
\address{School of Mathematical Sciences, University of Science and Technology of China, Hefei 230026,
	Anhui Province, China}
\email{xinan@ustc.edu.cn}

\author{Dekai Zhang}  
\address{Department of Mathematics, Shanghai University, Shanghai, 200444, China}
\email{dkzhang@shu.edu.cn}

\begin{abstract}
We study the exterior Dirichlet problem for homogeneous $k$-Hessian equation. The prescribed asymptotic behavior at infinity of the solution is zero if $k<\frac{n}{2}$, it is $\log|x|+O(1)$ if $k=\frac{n}{2}$ and  it is $|x|^{\frac{2k-n}{n}}+O(1)$ if $k>\frac{n}{2}$. By constructing smooth  solutions of approximating non-degenerate $k$-Hessian equations with uniform $C^{1,1}$-estimates, we prove the existence part. The uniqueness follows from the comparison theorem and thus  the  $C^{1,1}$- regularity of the solution of the homogeneous $k$-Hessian equation in the exterior domain is proved. We also prove a uniform positive lower bound of the gradient. As an implication of the $C^{1,1}$-estimates, we derive an almost monotonicity formula along the level set of the approximating solution. In particular, we get an weighted geometric inequality which is a natural generalization of the $k=1$ case.

\end{abstract}
\maketitle.
\tableofcontents 

\section{Introduction}
Let $u$ be a $C^2$ function and $\lambda=(\lambda_1,\cdots,\lambda_n)$ be the eigenvalues of $D^2u$, the k-Hessian operator is defined by
\begin{align}
	S_{k}(D^2u):=S_k(\lambda)=\sum\limits_{1\le i_1<\cdots i_k\le n}\lambda_{i_1}\cdots \lambda_{i_k},
	\end{align}
where $1\le k\le n$.
When $k=1$, $S_1(D^2 u)=\Delta u$. When $k=n$, $S_n(D^2 u)=\det D^2 u$.

Let $\Omega$ be a bounded smooth domain in $\mathbb{R}^n$, the Dirichlet problem for the $k$-Hessian equation is as follows
\begin{align}\label{khessian}
	\left\{\begin{aligned}
		S_k(D^2 u) =&f \qquad\text{in} \quad \Omega,\\
		u=&\varphi \qquad\text{on} \quad \partial\Omega,
		\end{aligned}
	\right.
	\end{align}
where $f$ and $\varphi$ are given smooth functions.
When $k=1$, the $k$-Hessian equation is the Poisson equation.
When $k=n$, it is the well known Monge-Amp\`ere equation.
\subsection{Some known results}
We briefly give some known results of the Dirichlet problem for the $k$-Hessian equation in the  nondegenerate case i.e. $f>0$ and in the degenerate cases i.e. $f\ge 0$. In general, the $k$-Hessian equation is a fully nonlinear equation.
\subsubsection{Results on bounded domains }
If $f>0$,
Caffarelli-Nirenberg-Spruck \cite{CNSIII} solved \eqref{khessian} in a bounded $(k-1)$-convex domain.  Guan \cite{guan1994cpde} solved \eqref{khessian} by only assuming the existence of a subsolution. The advantage of Guan's result is that there are no geometric restriction on the domain.

The Dirichlet problem in bounded domains of  degenerate fully nonlinear equations has been studied extensively.
For the Dirichlet problem of degenerate Monge-Amp\`ere equation in bounded convex domain, Caffarelli-Nirenberg-Spruck \cite{cns1986rmi} show the $C^{1,1}$ regularity for the homogeneous case i.e. $f\equiv 0$.
If  $f$  satisfies $f^{\frac{1}{n-1}}\in C^{1,1}$, Guan-Trudinger-Wang \cite{guantrudingerwang1999acta} proved the $C^{1,1}$ regularity, which is optimal by Wang's counterexample \cite{wang1995pams}.
The $C^{1,1}$ regularity problem of degenerate k-Hessian equation with Dirichlet boundary value in bounded $(k-1)$-convex domain was solved by Krylov\cite{krylov1989, krylov1994} and Ivochina-Trudinger-Wang\cite{itw2004cpde} (PDE's proof) with the assumption $f^{\frac{1}{k}}\in C^{1,1}$.
Dong \cite{dong2006cpde} studied the mixed Hessian equations.
\subsubsection{Results on unbounded domains }
The exterior Dirichlet problem for viscosity solutions of nondegenerate  fully nonlinear  equations has been studied extensively. The $C^0$ viscosity solution  for the  Monge-Amp\`ere equation: $\det D^2 u=1$ with  prescribed asymptotic behavior at infinity was solved by Caffarelli-Li\cite{cl2003cpam}.
 The related problem for the $k$-Hessian equation : $S_k(D^2 u)=1$ was proved by Bao-Li-Li \cite{baolili2014tams}. For the related results  on  other type nondegenerate fully nonlinear equations, one can see \cite{baoli2013, lili2018jde, libao2014jde, li2019tams}. Note that in these cases the regularity are only continuous.

 Li-Wang\cite{liwang2015dcds}
  proved the global $C^{k+2,\alpha}$ regularity of the homogeneous Monge-Amp\`ere equation: $\det(u_{ij})=0$
 in a strip region: $\mathbb R^{n-1}\times[0,1]$  by assuming that the boundary functions are locally uniformly convex and $C^{k,\alpha}$. Moreover, they gave a counterexample to show the necessity of the uniform convexity of the boundary functions.
\subsection{Motivation}
The motivation of this paper arises from  proving geometric inequality by establishing certain monotonicity formula on the level set of solutions in exterior domains. Another one comes from studying the regularity of extremal function of the complex Monge-Amp\`ere operator.
\subsubsection{Geometric inequalities}
One motivation for us to consider the exterior Dirichlet problem for the homogeneous $k$-Hessian equation comes from the following geometric inequalities:
\begin{align}\label{AF0718}
	\left(\frac{V_{n-l}(\Omega)}{V_{n-1-l}(B)}\right)^{\frac{l}{n-l}}\le\left(\frac{V_{n-k}(\Omega)}{V_{n-1-k}(B)}\right)^{\frac{l}{n-k}},
\end{align}
where $0\le l<k\le n$, $V_{n-k}(\Omega)=\int_{\partial\Omega}H_{k-1}(\kappa)dA$, $V_{-1}:=|\Omega|$ and $H_k$ is the $k$-Hessian operator of the principal curvature $\kappa=(\kappa_1,\cdots,\kappa_{n-1})$ of $\p\Omega$. $\eqref{AF0718}$ are called  Alexandrov-Fenchel inequalities. An open question is whether  \eqref{AF0718} holds for  $(k-1)$-convex domain $\Omega$ i.e. $H_m>0$ for $1\le m\le k-1$.

When $\Omega$ is $(k-1)$-convex and starshaped,
Guan-Li\cite{guanli2009adv} proved \eqref{AF0718} by the method of inverse curvature flows. If $\Omega$ is $k$-convex, Chang-Wang\cite{cw2013adv}, Qiu\cite{q2015ccm} proved the above inequalities when $l=0$ by the optimal transport method.

Very recently, by considering the exterior Dirichlet problem of the Laplace equation, Agostiniani-Mazzieri \cite{AM2020CVPDE} proved several geometric inequalities such as the  Willmore inequality. By studying the the exterior Dirichlet problem of the $p$-Laplacian equation,
Fogagnolo and Mazzieri and Pinamonti \cite{fmp2019ihp} showed the volumetric Minkowski inequality i.e. the Alexandrov-Fenchel inequality with $l=0$ and $k=2$ for smooth convex domains. Later, Agostiniani-Fogagnolo-Mazzieri  \cite{afm2022arma} removed the convexity assumption for the domain. The key point for them is to prove a monotonicity formula along the level set of the solution of the exterior Dirichlet problem for the $p$-Laplace equation.
%

 \subsubsection{Regularity problems of extremal functions}
  P. F. Guan.
  \cite{gpf2002am, gpf2010} proved the $C^{1,1}$ regularity of the homogeneous complex Monge-Amp\`ere equation in $U:=V_0\setminus V$ with $V=\cup_{i=1}^N V_{i}$, where  $V_0$ and $ V_i$ are strongly pseudoconvex and bounded smooth domains in a complex manifold ${M}^n$, $V$ is holomorphically convex subset of $\Omega_1$. Then he solved a conjecture of Chern-Levine-Nirenberg on the extended intrinsic norms.
   B. Guan \cite{gb2007imrn} proved the $C^{1,1}$ regularity of solutions of the exterior Dirichlet problem for the homogeneous complex Monge-Amp\`ere
   equation in $\mathbb{C}^{n}\setminus\bar V$ with $V={ (\cup_{i=1}^N V_i)}$, where $V_i$ are strongly pseudoconvex and bounded smooth domains and $V$ is a  holomorphically convex subset of $V_0$.
   If $V$ is strictly convex and smooth (analytic), the smooth (analytic) regularity of this problem was proved by Lempert \cite{lempert1985duke}
\subsection{Our main results}
In this paper, we consider the following exterior Dirichlet problem for the k-Hessian equation. For convenience,  we always assume $0\in \Omega$ and  there exists positive constants $r_0, {R_0}$ such that $B_{r_0}\subset\subset\Omega\subset B_{\frac {R_0} {2}}$, where $B_r$ and $B_{\frac {R_0} {2}}$ are balls centered at $0$ with radius $r$ and $\frac{R_0}{2}$ respectively.

\subsubsection{\emph{\textbf{Case1:}} $1\le k<\frac{n}{2}$}
Since the Green function in this case is $-|x|^{\frac{2k-n}{k}}$, we consider the $k$-Hessian equation when $k<\frac{n}{2}$ as follows
\begin{equation}\label{case1Equa1.1}
\left\{\begin{aligned} S_{k}(D^2 u)=&0 \qquad\text{in}\quad \Omega^c:=\mathbb{R}^n\setminus\Omega,\\
u=&-1\quad \text{on}\ \partial\Omega,\\
\lim\limits_{x\rightarrow\infty}u(x)=&0.\end{aligned}
\right.
\end{equation}

\begin{theorem}\label{main07201}
Assume $1\le k<\frac{n}{2}$. Let $\Omega$ be a smoothly convex domain in $\mathbb{R}^n$ and strictly $(k-1)$-convex. There exists a unique  $k$-convex solution $u\in C^{1,1}(\overline{\Omega^c})$ of the equation \eqref{case1Equa1.1}. Moreover, there exists uniform constant $C$ such that for any $x\in \Omega^c$  the following holds
\begin{align}\label{decay10720}
\left\{
\begin{aligned}
C^{-1}|x|^{-\frac{n-2k}{k}}\le& -u(x)\le C|x|^{-\frac{n-2k}{k}},\\
 C^{-1}|x|^{-\frac{n-k}{k}}\le&|Du|(x)\le C|x|^{-\frac{n-k}{k}},\\
|D^2u|(x)\le& C|x|^{-\frac{n}{k}},
\end{aligned}
\right.
\end{align}
\end{theorem}
where the $k$-convex solution is defined in Section 2 and we use the notation $\overline \Omega^c:=\mathbb R^n\setminus \Omega$.
\subsubsection{\emph{\textbf{Case2:}} $k>\frac{n}{2}$}
Since the Green function in this case is $|x|^{\frac{2k-n}{k}}$, we consider the $k$-Hessian equation when $k>\frac{n}{2}$ as follows
\begin{equation}\label{case2Equa1.2}
\left\{\begin{aligned} S_{k}(D^2 u)=&0 \qquad\qquad\ \ \ \text{in}\quad \Omega^c,\\
u=&1\qquad \qquad \
\  \  \text{on}\ \ \  \partial\Omega,\\
u(x)=&|x|^{\frac{2k-n}{k}}+O(1)\ \text{as}\ |x|\rightarrow \infty.
\end{aligned}
\right.
\end{equation}
\begin{theorem}\label{main07202}
Assume $ k>\frac{n}{2}$. Let $\Omega$ be a smoothly convex domain in $\mathbb{R}^n$ and strictly $(k-1)$-convex. There exists a unique  $k$-convex solution $u\in C^{1,1}(\overline{\Omega^c})$ of the equation \eqref{case2Equa1.2}. Moreover, there exists uniform constant $C$ such that for any $x\in \Omega^c$ the following holds
\begin{align}\label{decay20720}
\left\{
\begin{aligned}
|u(x)-|x|^{\frac{2k-n}{k}}|\le& C,\\
C^{-1}|x|^{\frac{k-n}{k}}\le |Du|(x)\le& C|x|^{\frac{k-n}{k}},\\
|D^2u|(x)\le& C|x|^{-\frac{n}{k}}.
\end{aligned}
\right.
\end{align}
\end{theorem}

\subsubsection{\emph{\textbf{Case3:}} $k=\frac{n}{2}$}
Since the Green function in this case is $\log|x|$, we  consider the $k$-Hessian equation when $k=\frac{n}{2}$ as follows
\begin{equation}\label{case3Equa1.3}
\left\{\begin{aligned} S_{\frac{n}{2}}(D^2 u)=&0 \ \ \qquad\text{in}\ \ \Omega^c,\\
u=&0\ \ \qquad \text{on}\ \partial\Omega,\\
u(x)=\log|x|&+O(1) \ \text{as}\ {|x|\rightarrow\infty}.
\end{aligned}
\right.
\end{equation}
\begin{theorem}\label{main07203}
Assume $k=\frac{n}{2}$. Let $\Omega$ be a smoothly convex domain in $\mathbb{R}^n$ and strictly $(k-1)$-convex. There exists a unique  $k$-convex solution $u\in C^{1,1}(\overline{\Omega^c})$ of the equation \eqref{case3Equa1.3}. Moreover, there exists uniform constant $C$ such that for any $x\in \Omega^c$ the following holds
\begin{align}\label{decay30720}
\left\{
\begin{aligned}
 |u(x)-\log|x||\le& C,\\
  C^{-1}|x|^{-1}\le|Du|(x)\le& C|x|^{-1},\\
|D^2u|(x)\le& C|x|^{-2}.
\end{aligned}
\right.
\end{align}
\end{theorem}
\subsection{Applications}
To solve the above problems, we  consider the following approximating equation
\begin{align*}
	\left\{
	\begin{aligned}
		 S_{k}(u^{\varepsilon})=&f^{\epsilon}\ \text{in}\ \Omega^{c}\\
		u^{\varepsilon}=&-1\ \text{if}\ k<\frac{n}{2}, u^{\varepsilon}=1 \ \text{if}\ k>\frac{n}{2}, u^{\varepsilon}=0, \ \text{if}\ k=\frac{n}{2}\ \text{on}\ \p\Omega.\\
		u^{\varepsilon}(x)\rightarrow& 0 \ \text{if}\ k<\frac{n}{2},
		u^{\varepsilon}(x)=|x|^{\frac{2k-n}{k}}+O(1)\ \text{if}\ k>\frac{n}{2},
		u^{\varepsilon}(x)=\log|x|+O(1)
		\ \text{if}\ k=\frac{n}{2}, |x|\rightarrow \infty.
		\end{aligned}
	\right.
	\end{align*}
where $f^{\varepsilon}=c_{n,k}\varepsilon^2(|x|^2+\varepsilon^2)^{-\frac{n}{2}-1}$ (see  the precise value of $c_{n,k}$ in Section 4).

$u^{\varepsilon}$ will be obtained by approximating solutions $u^{\varepsilon,R}$ defined on bounded domains: $\Omega_R:=B_R\setminus\overline\Omega$(see Section 4 for precise definition of $u^{\varepsilon,R}$).
The existence and uniqueness of the $k$-convex solution of $u^{\varepsilon,R}$ follows from B. Guan\cite{guan1994cpde} if we can construct a subsolution, which can be constructed since we assume $\Omega$ is convex.
The key point is to establish the uniform $C^{2}$ estimates for $u^{\varepsilon,R}$.

As an application of our $C^{2}$ estimates, we can prove an almost monotonicity formula along the level set of $u^{\varepsilon}$ (see Section 6). Consequently, we get  geometric inequalities of $\p\Omega$ as follows.
\begin{theorem}\label{geometric0725}
	Let $\Omega$ be a smoothly convex domain in $\mathbb{R}^n$ and strictly $(k-1)$-convex.
	\begin{enumerate}[{(i)}]
		\item
	Assume $1\le k<\frac{n}{2}$ and $b\ge\frac{k(n-k-1)}{n-k}$. Let $u$ be the unique $C^{1,1}$ solution in Theorem \ref{main07201}. We have
		\end{enumerate}
	\begin{align}\label{0727geometric}
		\int_{\p\Omega}{|Du|^{b+1}H_{k-1}}\le \frac{n-2k}{n-k}\int_{\p\Omega}{|Du|^{b}H_{k}},
	\end{align}
	\end{theorem}
where $H_m$ is the $m$-Hessian operator of the principal curvature $\kappa=(\kappa_1,\cdots,\kappa_{n-1})$ of $\p\Omega$.
	\begin{enumerate}[{(ii)}]
		\item
Assume $k=\frac{n}{2}$ and $b>\frac{n}{2}-1$. Let $u$ be the unique $C^{1,1}$ solution in Theorem \ref{main07203}. We have
\end{enumerate}
\begin{align}
	\int_{\p\Omega}|Du|^{b+1}H_{k-1}
	\le \int_{\p\Omega}|Du|^{b}H_{k}.
\end{align}
\begin{remark}
	When $k=1$, \eqref{0727geometric} was proved by Agostiniani- Mazzieri \cite{AM2020CVPDE}.
\end{remark}

In section 2, we give some preliminaries. In section 3, we solve the Dirichlet problem of degenerate $k$-Hessian equation in a ring domain. Section 4 is the main part of this paper. We show uniform $C^{1,1}$ estimate of the solution which is the limit of the solutions of nondegenerate $k$-Hessian equation. The key ingredient is to establish  uniform gradient estimates and uniform second order estimates. We use the idea of Chow-Wang\cite{cw2001cpam} to establish the uniform second order estimate.
Theorem \ref{main07201}, Theorem \ref{main07202} and Theorem \ref{main07203} will be proved in Section 5. In section 6, we prove an almost monotonicity formula along the level set of the approximating solution and thus prove Theorem \ref{geometric0725}.

\begin{align*}
	\end{align*}

	Part of  results in this paper has been reported by Xinan Ma at \emph{ 21w5139-Interaction Between Partial Differential Equations and
		Convex Geometry}  on  October 17th 2021
	and by Dekai Zhang at seminars at Xiamen University, on November 3th, 2021 and at Academy of Mathematics and Systems Science, CAS, on July 6th, 2022.
	
	Very recently (July 12th, 2022), when $k<\frac{n}{2}$, Xiao\cite{xiao2022} solved the exterior Dirichlet problem for the homogenous $k$-Hessian equations in which Xiao assumed the domain is strictly $(k-1)$-convex and  starshaped. For the case of $k<\frac{n}{2}$, our proof is different from Xiao's. We directly prove the uniform $C^2$ decay estimates for the approximating solutions.

\section{Preliminaries}
\subsection{k-convex solutions}
In this section we give the definition of $k$-convex functions and definition of $k$-convex solutions.

The $\Gamma_k$-cone is defined by
\begin{align}
	\Gamma_{k}:=\{\lambda\in \mathbb{R}^n|S_{i}(\lambda)>0, 1\le i\le k\}
\end{align}
Recall $S_k(\lambda):=\sum\limits_{1\le i_1< \cdots <i_k\le n}\lambda_{i_1}\cdots\lambda_{i_k}$.

One can find the concavity property of $S_k^{\frac{1}{k}}$  in \cite{CNSIII}.
\begin{lemma}\label{concavity}
	$S_k^{\frac{1}{k}}$ is a concave function in $\Gamma_k$. In particular, $\log S_k$ is concave in $\Gamma_k$.
	\end{lemma}
For more properties of the $k$-Hessian operator, one can see the Lecture notes by Wang \cite{wang2009}.
We following the definition by Trudinger-Wang \cite{trudingerwang1997tmna} to give the definition of $k$-convex functions.
\begin{definition}
	Let $U$ be a domain in $\mathbb R^n$.
	
(1). A function $u\in C^2 (U)$ is called  \emph{$k$-convex} (\emph{strictly $k$-convex}) if $\lambda(D^2 u)\in \overline \Gamma_{k}$ \\($\lambda(D^2 u)\in\Gamma_k$).
	
(2). A function $u\in C^{0}(U)$ is called $k$-convex in $U$ if there exists a sequence of  functions $\{u_i\}\subset C^{2}(U)$ such that   in any bounded subdomain $V\subset\subset U$, $u_i$ is $k$-convex and
converges uniformly to $u$.
\end{definition}
\begin{definition}
Let $\Omega$ be a  bounded domain in ${\mathbb R}^n$ and $\varphi\in C^0(\p\Omega)$. A function $u\in C^0(\Omega^c)$ is called a $k$-convex solution of the homogeneous $k$-Hessian equation
	\begin{align}\label{hkhessian0718}
		\left\{\begin{aligned}
			S_k(D^2 u) =&0 \qquad\text{in} \quad \Omega^c:=R^n\setminus\overline\Omega,\\
			u=&\varphi \qquad\text{on} \quad \partial\Omega,
		\end{aligned}
		\right.
	\end{align}

 if there exists a sequence of $k$-convex functions $\{u_m\}\subset C^2(\Omega^c)$
 converging in $C^0(\Omega^c)$ to $u$ with  $\{S_{k}(D^2 u)\}$ converging in $L^1_{loc}(\Omega^c)$ to $0$ and $u=\varphi$ on $\p\Omega$.
\end{definition}
We need the following comparison principle by Wang-Trudinger\cite{trudingerwang1997tmna} (see also \cite{trudinger1997cpde, urbas1990indiana}) to prove the uniqueness of our equations.
\begin{lemma}\label{comparison0718}
	Let $u, v$ be $k$-convex functions in a bounded smooth domain $U$ in $\mathbb R^n$ satisfying
	\begin{align}
		\left\{\begin{aligned}
			S_{k}(D^2 u)\ge& S_{k}(D^2 v)\quad \text{in}\ U,\\
			u\le& v\ \quad\qquad\ \ \text{on}\ \p U,
			\end{aligned}
		\right.
		\end{align}
	in the viscosity sense. Then $u\le v$ in $U$.
	\end{lemma}

\subsection{The existence of the subsolution}
\begin{definition}
	A $C^2$ domain $U$ is called $(k-1)$-convex (strictly $(k-1)$-convex)  if for any $x\in \p U$, the principal curvature $\kappa:=(\kappa_1,\cdots,\kappa)$ of $\p U$ at $x\in \p U$ satisfies  $\kappa\in\overline\Gamma_k$  ($\kappa\in\Gamma_k$).
\end{definition}
Note that a $C^2$ domain $U$ is $(n-1)$-convex if and only if $U$ is convex.
\begin{definition}
	Let $U$ be a smoothly bounded domain in $\mathbb R^n$.  $\Phi$ is called a defining function of $U$ if $U=\{x:\Phi(x)<0\}$, $\Phi|_{\p U}=0$ and $|D\Phi||_{\p U}=1$.
	\end{definition}
 Caffarelli-Nirenberg-Spruck \cite{CNSIII} proved the following.
\begin{lemma}
	Let $U$ be a smoothly and strictly $(k-1)$-convex bounded domain in $\mathbb R^n$. There exists a smoothly and  strictly $k$-convex defining function $\Phi$ on $\overline U$.
	\end{lemma}
We need the following  Lemma by Guan\cite{gpf2002am} to construct the subsolution of the k-Hessian equation in a ring.
\begin{lemma}\label{Guan2002}
Suppose that $U$ is a bounded smooth domain in $\mathbb{R}^n$. For $h, g\in C^m(U)$, $m\ge 2$, for all $\delta>0$, there is an $H\in C^m(U)$ such that

\begin{enumerate}[(1)]
\item
 $H\ge \max\{h, g\}$ \ \text{and}

	\begin{align*}
 H(x)=\left\{ {\begin{array}{*{20}c}
		{h(x), \quad   \text{if } \ h(x)-g(x)>\delta  }, \\
		g(x) , \ \quad \text{if } \ g(x)-h(x)>\delta;\\
\end{array}} \right.
\end{align*}
\item
{There exists}  $|t(x)|\le 1 $ {such that}

\begin{align*}
\left\{H_{ij}(x)\right\}\ge
 \left\{\frac{1+t(x)}{2}g_{ij}+\frac{1-t(x)}{2}h_{ij}\right\},\ \text{for all} \ x\in\left\{|g-h|<\delta\right\}.
 \end{align*}
\end{enumerate}
	\end{lemma}
By Lemma \ref{concavity}, we see $H$ is  $k$-convex if $f$ and $g$ are both $k$-convex.
\begin{lemma}\label{subu0720}
	Let $\Omega_0$ and $\Omega_1$ be smoothly and strictly $(k-1)$-convex domain in $\mathbb R^n$ with $\Omega_0\subset\subset\Omega_1$ . Assume that $\Omega_0$ is convex. Then there exists a strictly $k$-convex function $\underline{u}\in C^{\infty}(\overline U)$ with $U:=\Omega_1\setminus\overline\Omega_0$ satisfying
\begin{align}\label{underlineu07191}
	\left\{\begin{aligned}
		S_{k}(D^2\underline u)\ge& \epsilon_0, \qquad  \text{in} \ \ \overline U, \\
		\underline{u}=&\mathrm{{\tau_0}}\Phi^0,\ \text{near} \ \ \partial \Omega_0,\\
		\underline{u}=&1+K_1\Phi^1, \ \text{near}\ \partial \Omega_1,
	\end{aligned}
	\right.
\end{align}
where $\Phi^{i}$ is the defining function of $\Omega_i$, $\tau_0$ and $K_1$ are uniform constants.
\end{lemma}

\begin{proof}
	If $\Omega_0$ is $(k-1)$-convex and smooth,
Caffarelli-Nirenberg-Spruck \cite{CNSIII} constructed a  strictly $k$-convex  defining function $\Phi_0 \in C^{\infty}(\overline\Omega_0)$ satisfying
\begin{align}
	\left\{
	\begin{aligned}
	S_{k}(D^2 \Phi^0)\ge& \epsilon_0\quad \text{on}\ \overline\Omega_0,\\
	\Phi^0=&t_0^{-1}\left(e^{- t_0d(x)}-1\right) \ \text{near} \ \p\Omega_0,
	\end{aligned}
	\right.
	\end{align}
where $\epsilon_0, t_0$ are positive constants and $d(x)$ is the distance function from $x$ to $\p\Omega_0$.\\
 Since we also assume that $\Omega_0$ is convex, $d(x)$ is smooth in $\Omega_0^c$. Then  we can take $\Phi^{0}(x)=t_0^{-1}\left(e^{t_0d(x)}-1\right)$ for any $x\in \Omega_0^c$ and we still have
\begin{align}
	S_k(D^2 \Phi^0)\ge \epsilon_0 \ \text{in}\ \Omega_0^c.
\end{align}

Let $g=\tau_0\Phi^0$, $h=1+K_1\Phi^1$. By
 Lemma \ref{Guan2002} ( $\delta=\frac{1}{2}$), for $K_1>0$ sufficiently large, there exists a smooth function $\underline{u}$ satisfying \eqref{underlineu07191}. Indeed, define $\Omega_{t_1}=\{x\in \Omega_1: \Phi^{1}(x)<-t_1 \}$ with $t_1>0$. Then for $t_1$ small enough, $\Omega_0\subset\subset\Omega_{t_1}$ and  $\mathrm{dist}(\p\Omega_{t_1},\partial \Omega_0)>\frac{1}{2}\mathrm{dist}(\p\Omega_1,\partial \Omega_0)$. Let $\Omega_{\frac{t_1}{8}}=\{x\in \Omega_1: \Phi^{1}(x)<-\frac{t_1}{8} \}$.

For any $x\in \overline\Omega_{t_1}\setminus\Omega_0$, by choosing $K_1=2t_1^{-1}$ large enough, we have
$$g(x)-h(x)\ge-h(x)\ge -1+K_1t_1=1>\frac 12.$$
Then $\underline{u}=g=\tau_0\Phi^0$ in $\overline\Omega_{t_1}\setminus\Omega_{0}$.

For any $x\in \overline\Omega_1\setminus\Omega_{\frac{t_1}{8}}$, by choosing $\tau_0$  small enough, we have
$$h-g\ge 1-\frac{1}{4}-\tau_0|\Phi^0(x)|>\frac 12.$$
Then $\underline{u}=h=1+2t_1^{-1}\Phi^1$ in $\overline\Omega_1\setminus\Omega_{\frac{t_1}{8}}$. Moreover, by Lemma \ref{Guan2002}, $\underline{u}$ is strictly $k$-convex.
\end{proof}
\subsection{Level sets}
For any function $u$ on a domain $U$, we define the level set of $u$ with height $t$ as follows
\begin{align}
	S_t:=\{x\in U: u(x)=t\}.
\end{align}

Let $H_m(x)$ be the $m$-Hessian operator of the principal curvature $\kappa(x)=(\kappa_1,\cdots,\kappa_{n-1})$ of $x\in S_t$. 
We have the following useful formula which can be founded in \cite{bns2008arma}.
\begin{lemma}\label{07261}
	Let $u\in C^2(U)$ and $|Du|\neq 0$.  Then on $S_t$, for $1\le m\le n$, we have
	\begin{align*}
		H_{m-1}=&\frac{S_{m}^{ij}(D^2 u)u_iu_j}{|Du|^{m+1}},\\
		S_m(D^2u)=&H_m|Du|^{m}+S_{m}^{ij}u_iu_{l}u_{lj}|Du|^{-2},
	\end{align*}
where $S_m^{ij}:=\frac{\p S_k(D^2 u)}{\p u_{ij}}$ and the curvature is defined with respect to the upward normal as in \cite{bns2008arma}.
	In particular, if $u$ is $k$-convex (strictly $k$-convex), the level set $S_t$ is $(k-1)$-convex (strictly $k$-convex).
\end{lemma}
\section{The Dirichlet problem for the homogeneous $k$-Hessian equations in the ring}
In this section, we prove the existence of the Dirichlet problem of degenerate $k$-Hessian equation in a smooth ring.
\begin{align}\label{EquaRing}
\left\{ {\begin{array}{*{20}c}
   {S_k(D^2 u)=0, \quad   \text{in} \quad U:=\Omega_1 \setminus\overline\Omega_0  }, \\
   {u=0 , \ \quad \text{on} \quad\partial\Omega_0}, \\
   {u=1 , \ \quad \text{on} \quad\partial \Omega_1}.  \\
\end{array}} \right.\end{align}
We assume that $\Omega_1$ is smoothly and strictly $(k-1)$-convex domain and $\Omega_0$ is a smoothly  strictly $(k-1)$-convex and convex domain.
 Using Lemma \ref{subu0720}, there exists a smoothly and strictly $k$-convex subsolution $\underline u$ satisfying
\begin{align}\label{underlineu071911}
	\left\{\begin{aligned}
		S_{k}(D^2\underline u)\ge& \epsilon_0, \qquad  \text{in} \ \ \overline U, \\
		\underline{u}=&\mathrm{{\tau_0}}\Phi^0,\ \text{near} \ \ \partial \Omega_0,\\
		\underline{u}=&1+K_1\Phi^1, \ \text{near}\ \partial \Omega_1,
	\end{aligned}
	\right.
\end{align}
where $\mathrm{{\tau_0}},\mathrm{{K_1}}$ are positive constants and $\Phi^i$ are  defining functions of $\Omega_i$.
\begin{theorem}\label{ring0720}
	Let $\Omega_0, \Omega_1$ be smooth $(k-1)$-convex domain and assume that $\Omega_0$ is convex.
There exists a unique solution $u\in C^{1,1}(\overline U)$ of the equation \eqref{EquaRing}.
\end{theorem}

The uniqueness follows from the classical comparison theorem for $k$-convex solutions of $k$-Hessian equations. Next, we prove the existence and regularity of $k$-convex solution by approximation. Indeed, for every $0<\epsilon<\epsilon_0$, we consider the following problem
\begin{align}\label{ApprEquaRing3}
\left\{ {\begin{array}{*{20}c}
   {S_k(D^2 u^{\epsilon})=\epsilon, \quad   \text{in} \quad\Omega  }, \\
   {u^\epsilon=0 , \ \quad \text{on} \quad\partial\Omega_0}, \\
   {u^\epsilon=1 , \ \quad \text{on} \quad\partial \Omega_1}.  \\
\end{array}} \right.\end{align}

Since $u^{\epsilon}$ is a subsolution, by Guan \cite{guan1994cpde}, the above problem has a unique smooth solution $u^\epsilon$.

Next, we want to show the $C^{1,1}$  estimates are independent of $\epsilon$. Firstly, by maximum principal, $u^{\epsilon_1}\ge u^{\epsilon_2}$ for any $\epsilon_1\le \epsilon_2$. Thus $u^{0}:=\lim\limits_{\epsilon\rightarrow\infty}u^{\epsilon}$ exists. If we could prove uniform $C^{1,1}$ estimates, then   $u^{0}$ is the $C^{1,1}$ solution of  equation \eqref{EquaRing}.
\begin{theorem}
	Let $u^{\epsilon}$ be the smooth $k$-convex solution of \eqref{ApprEquaRing3}. Then there exists a uniform constant $C$ independent of $\varepsilon$ such that
	\begin{align*}
		|u^{\varepsilon}|_{C^{1,1}(\overline U)}\le& C.
	\end{align*}
\end{theorem}
\textbf{In the following subsections, for simplicity, we use $u$ instead of $u^{\epsilon}$.}
\subsection{$C^1$-estimates}
\begin{lemma}
There exists a uniform constant $C$ such that
\begin{align}
	|u|_{C^1(\overline U)}\le C.
	\end{align}
\end{lemma}
\begin{proof}
Let $h$ be the unique solution of the problem
\begin{align}\label{mainequation}
\left\{ {\begin{array}{*{20}c}
   { \Delta h=0,\quad \text{in} \quad U }, \\
   {\ h=0, \ \quad  \text{on} \quad\partial\Omega_0}, \\
   {\ h=1, \ \quad \text{on} \quad\partial \Omega_1}.  \\
\end{array}} \right.\end{align}
By the maximal principle, we have $\underline{u}\le u\le h$. This gives the uniform $C^{0}$ estimates.

Let $F^{ij}:=\frac{\partial }{\partial u_{ij}}\log S_k(D^2 u)$.
Since $F^{ij}(u_{\xi})_{ij}=0$
for any unit constant vector $\xi$, we have $\max\limits_{\overline U}|Du|
=\max\limits_{\partial U}|Du|.$
Since $\underline{u}\le u^{\varepsilon}\le h$ in $U$ and $\underline{u}= u^{\varepsilon}= h$ on $\p U$, we have
\begin{align*}
 h_{\nu}\le u_{\nu}\le& \underline{u}_{\nu}<0, \ \text{on}\ \partial\Omega_0\\
 h_{\nu}\ge u_{\nu}\ge& \underline{u}_{\nu}>0, \ \text{on}\ \partial\Omega_1,
\end{align*}
where $\nu$ is the unit normal vector of $\partial U$ (inner normal vector of $\p\Omega_0$).
Thus we have
\begin{align}
	\max\limits_{\overline U}|Du|
=\max\limits_{\partial U}|Du|\le C.
\end{align}
\end{proof}
\subsection{Second order estimates}
\begin{lemma}
	There exists a uniform constant $C$ such that
	\begin{align}
		\max_{\overline U}|D^2 u|\le C.
	\end{align}
\end{lemma}
\begin{proof}
 Since $F^{ij}u_{\xi\xi ij}=-F^{ij,kl}u_{\xi ij }u^{\varepsilon}_{\xi kl}\ge 0$, we have $\max\limits_{\overline\Omega}u_{\xi\xi}\le \max\limits_{\partial\Omega}u_{\xi\xi}$.
Thus we need to prove the second order estimate on the boundary $\partial U$.  Here we use the method by B. Guan \cite{guan1994cpde} and P. F. Guan\cite{gpf2002am} (see also \cite{guan2014duke}).

\emph{Tangential derivative estimates on $\partial U$}\\
For any fixed $x_0\in\partial U$, we choose the coordinate such that $x_0=0$, $\partial U\bigcap B_\delta(x_0)=(x', \rho(x'))$, $\rho(0)=0$ and $\nabla \rho(0)=0$.
Since $u(x', \rho(x'))=0$, we have
\begin{align*}
0=&u_\alpha(x', \rho(x'))+u_n(x', \rho(x'))\rho_\alpha(x'),\\
0=&u_{\alpha\beta}(0)+u_{\alpha n}(0)\rho_\beta(0)+u_{n\beta}(0)\rho_\alpha(0)+u_{nn}(0)
\rho_\alpha(0)\rho_\beta(0)+u_{n\beta}(0)\rho_{\alpha\beta}(0)\\
=&u_{\alpha\beta}(0)+u_{n}(0)\rho_{\alpha\beta}(0).
\end{align*}
Then we have $|u_{\alpha\beta}(0)|\le C \max\limits_{\partial \Omega}|Du|\le C$.

\emph{Tangential-normal derivative estimates on $\partial U$.}\\
 We use Guan's method \cite{guan1994cpde} (see slao \cite{guan2014duke}). Our barrier function here is simpler than  before since $u$ is constant on the boundary and the right hand side of the approximating equation is a sufficiently small constant $\epsilon$.

For any fixed $x_0\in\partial U$, we choose the coordinate such that $x_0=0$, $\partial U\bigcap B_\delta(x_0)=(x', \rho(x'))$,$\nabla \rho(0)=0$ and  $\rho(x')=\sum_{\alpha<n}\kappa_\alpha |x_\alpha|^2+O(|x'|^3)$. Consider the tangential operator  $T_{\alpha}=\partial_{\alpha}+\kappa_\alpha(x_{\alpha}\partial_n-x_n\partial_\alpha)$.

We will prove $ w=A_1(u-\underline{u})+A_2|x|^2\pm T_{\alpha}u\ge 0$ in $U_{\delta}:=B_{\delta}(0)\cap U$.
\\
Since $u-\underline{u}=0$ and $|T_\alpha u|\le C|x'|^2$ on $\partial U\cap B_{\delta }(0)$, we have
\begin{align*}
	w|_{\partial \Omega\cap B_{\delta}(0)}=A_2|x|^2-C|x|^2\ge 0,
	\end{align*}
where we require $A_2>C$.
Since $|T_{\alpha}u|\le C$ and $u\ge \underline{u}$, on $U\cap \p B_{\delta }(0)$, we have
\begin{align*}
	w|_{ \Omega\cap \partial B_{\delta}(0)}=A_2\delta^2-C> 0,
\end{align*}
where $A_2>2C\delta^{-2}$.
Thus we have $w\ge 0$ on $\partial U_{\delta}$.

Next we show $F^{ij}w_{ij}<0$ in
$U_{\delta}$. Indeed, recall $\underline{u}$ is $k$-convex and  $S_{k}(D^2 \underline{u})\ge \epsilon_0>0$, there exits $\tau_0>0$ sufficiently small depending only on $\epsilon_0$ and
 $|\underline{u}|_{C^2}$ such that $\tilde{\underline{u}}:=\underline{u}-\tau_0|x|^2$ is $k$-convex and  $S_{k}(D^2 \tilde{ \underline{u}})\ge \frac{\epsilon_0}{2}$.\\
  By the concavity of $\log S_k$, we have
 \begin{align*}
 	F^{ij}(u_{ij}-\tilde{\underline{u}}_{ij})\le &F(D^2u)-F(D^2\underline{u})\\
 	=&\log \epsilon-\log S_{k}(D^2\tilde{\underline{u}})\\
 	\le&\log \epsilon-\log \frac{\epsilon_0}{2}\\
 	<&0,
 	\end{align*}
 where we take $2\epsilon<\epsilon_0$. Thus we have
 \begin{align*}
 	F^{ij}(u-{\underline{u}})_{ij}=&
 	F^{ij}(u_{ij}-\tilde{\underline{u}}_{ij})-2\tau_0\mathcal{F}\\
 	<& -2\tau_0\mathcal{F},
 \end{align*}
where $\mathcal{F}=\sum_{i=1}^n F^{ii}$.
Then we have
\begin{align*}
	F^{ij}w_{ij}=&
	F^{ij}(A_1(u-\underline{u})+A_2|x|^2+T_{\alpha}u)_{ij}\\
	=&A_1F^{ij}(u-{\underline{u}})_{ij}+2A_2\mathcal{F}\\
	\le& -2A_1\tau_0\mathcal{F}+2A_2\mathcal{F}\\
	<&0,
	\end{align*}
where we use $F^{ij}(T_{\alpha}u)_{ij}=0$ and we take $A_1=\frac{A_2}{2\tau_0}$. Then we obtain $w\ge 0$ in $U_{\delta}$ and $w(0)=0$. Namely we have
\begin{align}
	|T_{\alpha} u|\le& A(u-\underline{u})+A_2|x|^2 \ \text{in}\ U_{\delta},\\
	(T_{\alpha} u)(0)=&0.
	\end{align}
This gives $|u_{\alpha n}(0)|\le C$.

\emph{Double normal derivative estimates on $\partial U$ }\\
	For any fixed $x_0\in\partial U$, we choose the coordinate such that $x_0=0$, $\partial U\bigcap B_r(x_0)=(x', \rho(x'))$ and $\nabla \rho(0)=0$.
	
\emph{\textbf{Case 1:} $x_0\in\partial \Omega_1$}

We have
\begin{align*}
	u_{\alpha\beta}(0)=-u_{n}(0)\rho_{\alpha\beta}(0)=|Du|(0)\rho_{\alpha\beta}(0).
	\end{align*}
Since $|Du|(0)\ge c>0$ on $\partial \Omega_1$ and $\Omega_1$ is $(k-1)$-convex, we have
\begin{align}\label{0717c1}
	S_{k}(u_{\alpha\beta}(0))\ge c^{k}S_{k}(\rho_{\alpha\beta}(0))\ge c_1>0.
	\end{align}

\emph{\textbf{Case 2:} $x_0\in\partial \Omega_0$}

Since $\frac{\partial u}{\partial \nu}\le \frac{\partial \underline{u}}{\partial \nu}=-|D\underline{u}|<0$ on $\partial\Omega_0$, we have $|Du|>|D\underline{u}|>a_0>0$ and then there exists a smooth function $g$ such that  $u=g\underline{u}$ near $\partial\Omega_0$. Since $u\ge \underline{u}>0$ in $U$, we have $g\ge 1$ near $\partial\Omega_0$. On the other hand, since $\underline{u}=0$ on $\partial \Omega_0$, we have for any $1\le \alpha, \beta\le n-1$, $u_{\alpha}(0)=\underline{u}_{\alpha}(0)=0$. Thus \begin{align*}u_{\alpha\beta}(0)=&g_{\alpha\beta}(0)\underline{u}(0)+g_\alpha(0) \underline{u}_\beta(0)+g_\beta(0)\underline{u}_\alpha(0)+g(0)\underline{u}_{\alpha\beta}(0)\notag\\
=&g(0)\tau\Phi^{0}_{\alpha\beta}(0),
\end{align*}
where we have used $\underline{u}=\tau_0\Phi^0$ near $\partial\Omega_0$.
Therefore
\begin{align}\label{0717c2} S_{k-1}(u_{\alpha\beta}(0))=g^{k-1}(0)\tau_0^{k-1}S_{k-1}
\left(\Phi^{0}_{\alpha\beta}(0)\right)\ge \tau_0^{k-1}C_n^{k-1}C_{n}^{\frac{k-1}{k}}
\min\limits_{\partial\Omega_0}S_{k}^{\frac{k-1}{k}}
\left(D^2\Phi^{0}\right):=c_2>0.
\end{align}
 Let $c_0=\min\{c_1, c_2\}$ (see \eqref{0717c1} and \eqref{0717c2}),  we have
\begin{align*}
u_{nn}c_0\le u_{nn}(0)S_{k-1}(u_{\alpha\beta}(0))=&S_k(D^2 u(0))-S_{k}(u_{\alpha\beta}(0))
+\sum_{i=1}^{n-1}u_{in}^2S_{k-2}(u_{\alpha\beta})\\
\le& C.
\end{align*}
Then we obtain
\begin{align*}
	u_{nn}(0)\le C,
	\end{align*}
where $C$ is a uniform constant.
  On the other hand, $u_{nn}(0)\ge \sum\limits_{i=1}^{n-1}u_{\alpha\alpha}(0)\ge -C$. In conclusion, we have $|u_{nn}(0)|\le C$.

  In conclusion, we get the uniform $C^2$ estimate.
\end{proof}

\subsection{Proof of Theorem \ref{ring0720}}

The uniqueness follows from the comparison principal for $k$-convex solutions of $k$-Hessian equations in Lemma \ref{comparison0718} by Wang-Trudinger\cite{trudingerwang1997tmna} (see also \cite{trudinger1997cpde, urbas1990indiana})..

For the existence part,  since $u^{\epsilon}$ is increasing on $\epsilon$,  $u^{0}:=\lim\limits_{\epsilon\rightarrow 0}u^{\epsilon}$ exits. Since  $|u^{\epsilon}|_{C^2(\overline U)}\le C$, there exists a subsequence $u^{\epsilon_i}$  converges to $u^0 $ in $C^{1,\alpha}$ on $\overline U$ and $u^0\in C^{1,1}(\overline U)$.

\section{Solving the approximating equation in $\Omega_R:=B_R\setminus \Omega$.}
We always assume $\Omega$ is a smoothly convex domain and strictly $(k-1)$-convex. Recall that we always assume $B_r\subset\subset\Omega\subset\subset B_{\frac{R_0}{2}}$.
\subsection{Case 1: $k<\frac{n}{2}$}
Since the Green function in this case is $-|x|^{\frac{2k-n}{k}}$, we want to solve the following $k$-Hessian equation .
\begin{equation}\label{case1Equa}
\left\{\begin{aligned}S_{k}(D^2 u)=&0 \qquad\text{in}\quad \Omega^c:=\mathbb{R}^n\setminus\Omega,\\
u=&-1\ \ \ \text{on}\ \ \partial\Omega,\\
\lim\limits_{x\rightarrow\infty}u(x)=&0.
\end{aligned}
\right.
\end{equation}


Define $w^{1,\varepsilon}:=-\left({R_0^2+\varepsilon^2}
\right)^{\frac{n-2k}{2k}}
\left(|x|^2+\varepsilon^2\right)^{-\frac{n-2k}{2k}}$.
We have
\begin{align*}
f^{1,\varepsilon}:=S_{k}(D^2w^{1,\varepsilon})
=&C_{n}^k\left(\frac{k}{n-2k}\right)^k
(R_0^2+\varepsilon^2)^{\frac{n-2k}{2}}
(|x|^2+\varepsilon^2)^{-\frac{n}{2}-1}\varepsilon^2.
\end{align*}
For $\varepsilon$ small enough, we  can construct a smoothly strictly $k$-convex function $\underline {u}^{1,\varepsilon}$ as follows
\begin{lemma}
For any $\varepsilon\in (0,\frac{R_0}{3})$ small enough,	there exists a  strictly $k$-convex function $\underline {u}^{1,\epsilon}\in C^{\infty}{(\mathbb R^n \setminus\Omega)}$ satisfying
\begin{align}
\underline{u}^{1,\varepsilon}=&
\left\{\begin{aligned}\label{subu107191}
& \tau_0(e^{d(x)}-1)-1
, \ \ \qquad\qquad \text{in}\ B_{\frac{2}{3}R_0}\setminus \Omega, \\
&
w^{1,\varepsilon} \qquad\qquad\qquad\qquad\ \quad\ \text{in}\ B_{2R_0}^c,
\end{aligned}
\right.\\
\underline{u}^{1,\varepsilon}\ge& \max\left\{ w^{1,\varepsilon}, \tau_0(e^{d(x)}-1)-1\right\} \ \ \ \text{in}\ B_{2R_0}\setminus B_{\frac{2}{3}R_0},\notag\\
S_k(D^2 \underline{u}^{1,\varepsilon})\ge& f^{1,\varepsilon}, \qquad\qquad\qquad\qquad\qquad\text{in}\  \Omega^c,\notag
\end{align}
where $\tau_0=2^{-\frac {n}{2k}}\frac{2^{\frac{n-2k}{2k}}-1}{e^{3R_0}-1}>0$ since $k<\frac{n}{2}$.
\end{lemma}
\begin{proof}
We apply Lemma \ref{Guan2002} by taking $U=B_{3R_0}\setminus \overline\Omega$, $h=w^{1,\varepsilon}$, $g=\tau_0(e^{d(x)}-1)-1$ and $\delta=2^{-\frac {n}{2k}}({2^{\frac{n-2k}{2k}}}-1)$ to get a function $u^{1,\varepsilon}\in C^{\infty}({\overline U})$ which  is strictly $k$-convex and satisfies
\begin{align*}
	\underline{u}^{1,\varepsilon}\ge& \max\left\{ w^{1,\varepsilon}, \tau_0(e^{d(x)}-1)-1\right\} \ \ \ \text{in}\ B_{2R_0}\setminus B_{\frac{2}{3}R_0}\\
	S_k(D^2\underline{u}^{1,\varepsilon} )\ge& f^{1,\varepsilon} \ \text{in}\ U.
\end{align*}

Next we prove \eqref{subu107191}.
When $x\in B_{3R_0}\setminus B_{2R_0}$,  for $\varepsilon<R_0$,
\begin{align*}
h(x)-g(x)=&w^{1,\varepsilon}(x)-\tau_0(e^{d(x)}-1)+1\\
\ge &-\left({R_0^2+\varepsilon^2}
\right)^{\frac{n-2k}{2k}}\left({4R_0^2+\varepsilon^2}
\right)^{-\frac{n-2k}{2k}}-\tau_0(e^{3R_0}-1)+1\\
>&1-{2^{-\frac{n-2k}{2k}}}-\tau_0(e^{3R_0}-1)\\
>&\frac{1}{2}(1-{2^{-\frac{n-2k}{2k}}})\\
=:&\delta>0,
\end{align*}
where $\delta>0$ since $k<\frac{n}{2}$.\\
When $x\in  B_{\frac{2}{3}R_0}\setminus\Omega$, since $\varepsilon<\frac{R_0}{3}$, we have
\begin{align*}
g(x)-h(x)\ge &-w^{1,\varepsilon}(x)-1\\
\ge&\left({R_0^2+\varepsilon^2}
\right)^{\frac{n-2k}{2k}}\left({\frac{4}{9}R_0^2+\varepsilon^2}
\right)^{-\frac{n-2k}{2k}}-1\\
 \ge& {2^{\frac{n-2k}{2k}}}-1>\delta.
\end{align*}

We can finish the proof by extending the domain of  $u^{1,\epsilon}$  to $\Omega^c$ by taking $u^{1,\epsilon}=w^{1,\epsilon}$ in $B_{3R_0}^c$.
\end{proof}
Now for any $\varepsilon\in(0, \varepsilon_0)$ and $R\in(K_0(R_0+1), +\infty)$ with $\varepsilon_0$ small enough and $K_0$ large enough, we consider the approximating equation
\begin{equation}\label{case1EquaAppr}
\left\{\begin{aligned}S_{k}(D^2 u)=& f^{1,\varepsilon} \ \quad\text{in}\quad \Omega_R=B_R\setminus\Omega,\\
u=& -1\ \text{on}\quad \partial\Omega,\\
u(x)=& \underline u^{1,\varepsilon} \quad\text{on} \quad\partial B_{R}.
\end{aligned}
\right.
\end{equation}
Since $u^{1,\varepsilon}$ is a subsolution, by Guan \cite{guan1994cpde}, equation \eqref{case1EquaAppr} has a strictly $k$-convex solution $u^{\varepsilon,R}\in C^{\infty}(\overline \Omega_R)$. Our goal is to establish uniform $C^2$ estimates of $u^{\varepsilon,R}$, which are independent of $\varepsilon$ and $R$.
We prove the following
\begin{theorem}\label{apu10720}
Assume $1\le k<\frac{n}{2}$. For every sufficiently small $\varepsilon$ and sufficiently large $R$,    $u^{\varepsilon, R}$ satisfies
\begin{align*}
\left\{
\begin{aligned}
C^{-1}|x|^{-\frac{n-2k}{k}}\le& -u^{\varepsilon, R}(x)\le C|x|^{-\frac{n-2k}{k}},\\
C^{-1}|x|^{-\frac{n-k}{k}}\le|Du^{\varepsilon, R}|(x)\le& C|x|^{-\frac{n-k}{k}},\\
|D^2u^{\varepsilon, R}|(x)\le& C|x|^{-\frac{n}{k}},
\end{aligned}
\right.
\end{align*}
where $C$ is a uniform constant which is independent of $\varepsilon$ and $R$.
\end{theorem}

\subsection{{Case 2:} $k>\frac{n}{2}$}
Since the Green function in this case is $|x|^{\frac{2k-n}{k}}$, we want to solve the $k$-Hessian equation as follows
\begin{equation}\label{case2Equa}
\left\{\begin{aligned}S_{k}(D^2 u)=&0 \quad\text{in}\ \ \  \Omega^c,\\
u=&1\ \ \ \ \text{on}\ \partial\Omega,\\
u(x)=&|x|^{\frac{2k-n}{k}}+O(1) \ \text{as}\ {x\rightarrow\infty}.
\end{aligned}
\right.
\end{equation}

\subsubsection{The approximating equation }
Define $w^{2,\varepsilon}:=\left(
{|x|^2+\varepsilon^2}\right)^{\frac{2k-n}{2k}}-\left(
{R_0^2+\varepsilon^2}\right)^{\frac{2k-n}{2k}}+1$ and we have
\begin{align*}
	f^{2,\varepsilon}:=S_{k}(D^2w^{2,\varepsilon})
	=&C_{n}^k\left(\frac{k}{2k-n}\right)^k
	(|x|^2+\varepsilon^2)^{-\frac{n}{2}-1}\varepsilon^2.
\end{align*}

We construct a smoothly and strictly $k$-convex function $\underline {u}^{2,\varepsilon}$ as follows
\begin{lemma}
For any $\varepsilon\in (0,\frac{R_0}{3})$ small enough, there exists a  strictly $k$-convex function $\underline {u}^{2,\varepsilon}\in C^{\infty}{(\mathbb R^n \setminus\Omega)}$ satisfying
	\begin{align}
		\underline{u}^{2,\epsilon}=&
		\left\{\begin{aligned}
			\label{subu207191}
			& \tau_0(e^{d(x)}-1)+1
			, \quad\qquad\qquad \text{in}\ B_{\frac{2}{3}R_0}\setminus \overline\Omega, \\
			&
			w^{2,\varepsilon} \qquad\qquad\qquad\qquad\ \quad\ \text{in}\ B_{2R_0}^c,
		\end{aligned}
		\right.\\
		\underline{u}^{2,\varepsilon}\ge& \max\left\{ w^{2,\varepsilon}, \tau_0(e^{d(x)}-1)+1\right\} \ \ \ \text{in}\ B_{2R_0}\setminus B_{\frac{2}{3}R_0},\notag\\
		S_k(D^2 \underline{u}^{2,\epsilon})\ge& f^{2,\varepsilon}, \qquad\qquad\qquad\qquad\qquad\text{in}\  \Omega^c,\notag
	\end{align}
	where $\tau_0=\frac12 R_0^{\frac {2k-n}{k}}\frac{2^{\frac{2k-n}{2k}}-1}{e^{3R_0}-1}>0$ since $k>\frac{n}{2}$.
\end{lemma}
\begin{proof}
	We apply Lemma \ref{Guan2002} by taking $U=B_{3R_0}\setminus \overline\Omega$, $h=w^{2,\varepsilon}$, $g=\tau_0(e^{d(x)}-1)-1$ and $\delta=\frac12R_0^{\frac {2k-n}{k}}({2^{\frac{2k-n}{2k}}}-1)$ to get a function $u^{2,\varepsilon}\in C^{\infty}({\overline U})$ which  is strictly $k$-convex and satisfies
	\begin{align*}
		\underline{u}^{2,\varepsilon}\ge& \max\left\{ w^{2,\varepsilon}, \tau_0(e^{d(x)}-1)-1\right\} \ \ \ \text{in}\ B_{2R_0}\setminus B_{\frac{2}{3}R_0}\\
		S_k(D^2\underline{u}^{2,\varepsilon} )\ge& f^{2,\varepsilon} \ \text{in}\ U.
	\end{align*}
	
	Next we prove \eqref{subu207191}.
	When $x\in B_{3R_0}\setminus B_{2R_0}$,  for $\varepsilon<R_0$,
	\begin{align*}
		h(x)-g(x)=&w^{2,\varepsilon}(x)-\tau_0(e^{d(x)}-1)-1\\
		\ge &\left({4R_0^2+\varepsilon^2}
		\right)^{\frac{2k-n}{2k}}-\left({R_0^2+\varepsilon^2}
		\right)^{\frac{n-2k}{2k}}-\tau_0(e^{3R_0}-1)\\
		>&R_0^{\frac {2k-n}{k}}({2^{\frac{2k-n}{2k}}-1})-\tau_0(e^{3R_0}-1)\\
		>&\frac{1}{2}R_0^{\frac {2k-n}{k}}({2^{\frac{2k-n}{2k}}-1})\\
		=:&\delta>0,
	\end{align*}
	where we choose $\tau_0=\frac12 R_0^{\frac {2k-n}{k}}\frac{2^{\frac{2k-n}{2k}}-1}{e^{3R_0}-1}>0$ and $\delta>0$ since $k>\frac{n}{2}$.\\
	When $x\in  B_{\frac{2}{3}R_0}\setminus\Omega$,
	\begin{align*}
		g(x)-h(x)\ge &-w^{2,\varepsilon}(x)-1\\
		\ge&
		(R_0^2+\varepsilon^2)^{\frac{2k-n}{2k}}-(\frac 4 9R_0^2+\varepsilon^2)^{\frac{2k-n}{2k}}\\
		\ge&2^{-\frac{2k-n}{2k}}R_0^{\frac {2k-n}{k}}({2^{\frac{2k-n}{2k}}}-1)
		>\delta.
	\end{align*}
	
	Then we  finish the proof by extending the domain of  $u^{2,\epsilon}$  to $\Omega^c$ by taking $u^{2,\epsilon}=w^{2,\epsilon}$ in $B_{3R_0}^c$.
\end{proof}

We consider the  approximating equation as follows

\begin{align}\label{case2EquaAppr}
	\left\{ {\begin{array}{*{20}c}
			{S_k(D^2 u)=f^{2,\varepsilon}  , \ \text{in}\ \  B_R \setminus\overline\Omega,}\\
			{u=1,\  \text{on} \ \ \ \partial\Omega} , \\
			{u=\underline{u}^{2,\varepsilon}, \ \text{on}\ \  \  \partial B_R }.  \\
	\end{array}} \right.
\end{align}
Since $u^{2,\varepsilon}$ is a subsolution, by Guan \cite{guan1994cpde}, equation \eqref{case2EquaAppr} has a strictly $k$-convex solution $u^{\varepsilon,R}\in C^{\infty}(\overline \Omega_R)$. Our goal is to establish uniform $C^2$ estimates of $u^{\varepsilon,R}$, which are independent of $\varepsilon$ and $R$.

We prove the following
\begin{theorem}\label{apu20720}
 Assume $ k>\frac{n}{2}$. For every sufficiently small $\varepsilon$ and sufficiently large $R$,    $u^{\varepsilon, R}$ satisfying
	\begin{align*}
		\left\{
		\begin{aligned}
			| u^{\varepsilon, R}(x)-|x|^{\frac{2k-n}{k}}|\le  C,\\
		 C^{-1}|x|^{\frac{k-n}{k}}\le	|Du^{\varepsilon, R}|(x)\le& C|x|^{\frac{k-n}{k}},\\
			|D^2u^{\varepsilon, R}|(x)\le& C|x|^{-\frac{n}{k}},
		\end{aligned}
		\right.
	\end{align*}
	where $C$ is a uniform constant which is independent of $\varepsilon$ and $R$.
\end{theorem}

\subsection{{Case 3:} $k=\frac{n}{2}$}
Since the Green function in this case is $\log |x|$, we want to solve the $k$-Hessian equation as follows
\begin{equation}\label{case2Equa}
\left\{\begin{aligned}S_{\frac{n}{2}}(D^2 u)=&0 \quad\text{in}\quad \Omega^c,\\
u=&0\ \ \ \text{on}\ \partial\Omega,\\
u(x)=&\log |x|+O(1) \ \text{as}\ {x\rightarrow\infty}.
\end{aligned}
\right.
\end{equation}

\subsubsection{The approximating equation }
Define $w^{3,\varepsilon}:=\frac{1}{2}\log\frac{|x|^2+\varepsilon^2}{R_0^2+\varepsilon^2}$ and we have
\begin{align}
	f^{3,\varepsilon}:=S_k(D^2w^{3,\varepsilon})=2^{k+1}C_{n-1}^{\frac{n}{2}-1}\varepsilon^2(|x|^2+\varepsilon^2)^{-\frac n2-1}
	\end{align}

We construct a smoothly and strictly $k$-convex function $\underline {u}^{3,\varepsilon}$ as follows
\begin{lemma}
For any $\varepsilon\in (0,\frac{R_0}{3})$,	there exists a  strictly $k$-convex function $\underline {u}^{3,\varepsilon}\in C^{\infty}{(\mathbb R^n \setminus\Omega)}$ satisfying
	\begin{align}
		\underline{u}^{3,\varepsilon}=&
		\left\{\begin{aligned}\label{subu307191}
			& \tau_0(e^{d(x)}-1)
			\quad\qquad\qquad \text{in}\ B_{\frac{2}{3}R_0}\setminus \Omega, \\
			&
			w^{3,\varepsilon} \qquad\qquad\qquad\qquad\ \quad\ \text{in}\ B_{2R_0}^c,
		\end{aligned}
		\right.\\
		\underline{u}^{3,\varepsilon}\ge& \max\left\{ w^{3,\varepsilon}, \tau_0(e^{d(x)}-1)\right\} \ \ \ \text{in}\ B_{2R_0}\setminus B_{\frac{2}{3}R_0},\notag\\
		S_k(D^2 \underline{u}^{3,\varepsilon})\ge& f^{3,\varepsilon} \qquad\qquad\qquad\qquad\qquad\text{in}\  \Omega^c,\notag
	\end{align}
	where $\tau_0=\frac{1}{4}\frac{\log 2}{e^{3R_0}-1}$.
\end{lemma}
\begin{proof}
	We apply Lemma \ref{Guan2002} by taking $U=B_{3R_0}\setminus \overline\Omega$, $h=w^{2,\varepsilon}$, $g=\tau_0(e^{d(x)}-1)$ and $\delta=\frac14\log2$ to get a function $u^{2,\varepsilon}\in C^{\infty}({\overline U})$ which  is strictly $k$-convex and satisfies
	\begin{align*}
		\underline{u}^{3,\varepsilon}\ge& \max\left\{ w^{3,\varepsilon}, \tau_0(e^{d(x)}-1)\right\} \ \ \ \text{in}\ B_{2R_0}\setminus B_{\frac{2}{3}R_0}\\
		S_k(D^2u^{3,\varepsilon} )\ge& f^{3,\varepsilon} \ \text{in}\ U.
	\end{align*}
	
	Next we prove \eqref{subu307191}.
	When $x\in B_{3R_0}\setminus B_{2R_0}$,
	\begin{align*}
		h(x)-g(x)=&w^{3,\varepsilon}(x)-\tau_0(e^{d(x)}-1)\\
		\ge &\frac{1}{2}\log\frac{4R_0^2+\varepsilon^2}{R_0^2+\varepsilon^2}
		-\tau_0(e^{3R_0}-1)\\
		>&\frac{1}{2}\log2-\tau_0(e^{3R_0}-1)\\
		>&\frac{1}{4}\log 2
		=:\delta,
	\end{align*}
	where we use $\varepsilon<R_0$ and we choose $\tau_0=\frac{1}{4}\frac{\log 2}{e^{3R_0}-1}$.\\
	When $x\in  B_{\frac{2}{3}R_0}\setminus\Omega$, since $\varepsilon<\frac{1}{3}R_0$, we have
	\begin{align*}
		g(x)-h(x)\ge &-w^{2,\varepsilon}(x)\\
		\ge&
		\frac{1}{2}\log\frac{R_0^2+\varepsilon^2}{\frac49R_0^2+\varepsilon^2}\\
		\ge&\frac{1}{2}\log 2
		>\delta.
	\end{align*}
	
	Then we  finish the proof by extending the domain of  $u^{2,\epsilon}$  to $\Omega^c$ by taking $u^{2,\epsilon}=w^{2,\epsilon}$ in $B_{3R_0}^c$.
\end{proof}
We consider the  approximating equation as follows
\begin{align}\label{case3EquaAppr}
	\left\{ {\begin{array}{*{20}c}
			{S_k(D^2 u)=f^{3,\varepsilon}  , \ \text{in}\ \  B_R \setminus\overline\Omega,}\\
			{u=0,\  \text{on} \ \ \ \partial\Omega} , \\
			{u=\underline{u}^{3,\varepsilon}, \ \text{on}\ \  \  \partial B_R }.  \\
	\end{array}} \right.
\end{align}
Since $u^{3,\varepsilon}$ is a subsolution, by Guan \cite{guan1994cpde}, equation \eqref{case2EquaAppr} has a strictly $k$-convex solution $u^{\varepsilon,R}\in C^{\infty}(\overline \Omega_R)$.
%

We prove the following
\begin{theorem}\label{apu30720}
	Assume $ k=\frac{n}{2}$.  For every sufficiently small $\varepsilon$ and sufficiently large $R$,    $u^{\varepsilon, R}$ satisfies
	\begin{align*}
		\left\{
		\begin{aligned}
			 |u^{\varepsilon, R}(x)-\log|x||\le& C,\\
		C^{-1}|x|^{-1}\le	|Du^{\varepsilon, R}|(x)\le& C|x|^{-1},\\
			|D^2u^{\varepsilon, R}|(x)\le& C|x|^{-2},
		\end{aligned}
		\right.
	\end{align*}
	where $C$ is a uniform constant which is independent of $\varepsilon$ and $R$.
	\end{theorem}
In the next subsections, we will prove uniform $C^{2}$-estimates of solutions of equations \eqref{case1EquaAppr}, \eqref{case2EquaAppr} and \eqref{case3EquaAppr}. The key point is that these estimates are independent of $\varepsilon$ and $R$.
\subsection{$C^0$ estimates}
\subsubsection{Case 1: $k<\frac{n}{2}$}
We fisrt prove $u^{\epsilon,R}$ is increasing with $R$. Indeed, since  for any $\widetilde R>R\ge 100(R_0+1)$, we have

\begin{align*}
\left\{\begin{aligned}
S_k(D^2u^{\varepsilon,R})=&S_k(D^2u^{\varepsilon,\widetilde{R}})=f^{1,\varepsilon}
\ \text{in}\ \Omega_R\\
u^{\varepsilon,R}=&u^{\varepsilon,\widetilde R}=-1 \ \text{on}\  \partial \Omega,\\
u^{\varepsilon,R}=&\underline{u}^{\varepsilon}\le u^{\varepsilon,\widetilde R}\  \text{on} \ \partial B_R
\end{aligned}
\right.
\end{align*}
 Applying the maximum principle in $\Omega_R$, we have
 \begin{align}\label{C0Esti1}
 u^{\varepsilon,R}\le {u}^{\varepsilon,\tilde R}.
 \end{align}
 Since
\begin{align*}
\left\{\begin{aligned}
S_k(D^2u^{\varepsilon,R})=f^{1,\varepsilon}>
0=S_k\left(D^2\left({ r_0^{ \frac{n-2k}{k} }}{ |x|^{-\frac{n-2k}{k}} }\right)\right),
\ \text{in}\ \Omega_{\tilde{R}}\\
u^{\varepsilon,\tilde R}=-1<-{r_0^{\frac{n-2k}{k}}}{|x|^{-\frac{n-2k}{k}}}, \ \text{on}\  \partial \Omega \ (\textbf{\text{For}} \ B_{r_0}\subset\subset\Omega),\\
 u^{\varepsilon,\tilde R} =-\left(\frac{R_0^2+\varepsilon^2}{\widetilde{R} ^2+\varepsilon^2}\right)^{\frac{n-2k}{2k}}<-{r_0^{\frac{n-2k}{k}}}
 {\widetilde{R} ^{-\frac{n-2k}{k}}},\  \text{on} \ \partial B_{\tilde {R}}
\end{aligned}
\right.
\end{align*}
Applying the maximum principle in $\Omega_{\widetilde R}$, we have
\begin{align}\label{C0Esti2}
u^{\varepsilon,\tilde R}<-{r_0^{\frac{n-2k}{k}}}{|x|^{-\frac{n-2k}{k}}}\quad \text{in}\quad \Omega_{\widetilde R}.
\end{align}
Then by \eqref{C0Esti1} and \eqref{C0Esti2}, for any $\widetilde R>R\ge 100(R_0+1)$,
\begin{align}\label{C0Esti}
u^{\varepsilon, R}\le u^{\varepsilon,\tilde R}<-{r_0^{\frac{n-2k}{k}}}{|x|^{-\frac{n-2k}{k}}}\quad \text{in}\quad \Omega_{R}.
\end{align}
On the other hand, for any $x\in \Omega_R$, we have $u^{\varepsilon, R}(x)\ge \underline{u}^{1,\varepsilon}(x)\ge w^{1,\varepsilon}=-(R_0^2+\varepsilon^2)^{\frac{n-2k}{2k}}
(|x|^2+\varepsilon^2)^{-\frac{n-2k}{2k}}$. Thus  when $k<\frac{n}{2}$, for any $x\in \Omega_R$,
\begin{align*}
r_0^{\frac{n-2k}{k}}|x|^{-\frac{n-2k}{k}}\le -u^{\varepsilon,R}(x)\le R_0^{\frac{n-2k}{k}}|x|^{-\frac{n-2k}{k}}.
\end{align*}

\subsubsection{Case 2: $k>\frac{n}{2}$}
Firstly, we have for any $x\in \Omega_R$, $u^{\varepsilon, R}(x)\ge \underline{u}^{2,\varepsilon}\ge w^{2,\varepsilon}=
(|x|^2+\varepsilon^2)^{\frac{2k-n}{2k}}-(R_0^2+\varepsilon^2)^{\frac{2k-n}{2k}}+1$ and this gives the lower bound of $u^{\varepsilon, R}$.

Since ${|x|^{\frac{2k-n}{k}}}-{r_0^{\frac{2k-n}{k}}}+1$ is the upper barrier of $u^{\varepsilon, R}$ in $\Omega_R$, we get ${u}^{\varepsilon, R}-|x|^{\frac{2k-n}{k}}\le C$.
\subsubsection{Case 3: $k=\frac{n}{2}$}
The proof is similar as that in Case 2.\\
\subsection{Gradient estimates}
In this subsection, we prove the global gradient estimate. The key point is that the estimate here does not depend on $\varepsilon$ and $R$. We also prove that the positive lower bound of the  gradient of the solution  and thus the level set of the solution is compact.
\subsubsection\textbf{\emph{{Reducing global gradient estimates to boundary gradient estimates}}}
This part is the key part of gradient estimates. The point in here is that the gradient estimate is independent of the approximating process.  This estimates is motivated by B. Guan \cite{gb2007imrn}.
\begin{theorem}
Let $u$ be the solution of the approximating equation  \eqref{case1EquaAppr}, \eqref{case2EquaAppr} or \eqref{case3EquaAppr}. Denote by
\begin{align}\label{threecase}
P=
\left\{ {\begin{array}{*{20}c}
   {|Du|^2e^{2u}, \quad \quad  \ k=\frac{n}{2}  }, \\
   {|Du|^2 u^{\frac{2(n-k)}{2k-n}} , \ \quad\  k>\frac{n}{2} }, \\
   {|Du|^2 (-u)^{-\frac{2(n-k)}{n-2k}}, \ k<\frac{n}{2} }.  \\
\end{array}} \right.
\end{align}
then we have the following gradient estimate
\begin{align}\label{Gradientestimate}
\max\limits_{B_R\setminus\Omega} P\le
\left\{ {\begin{array}{*{20}c}
   {
   \max \left\{\max\limits_{B_R\setminus\Omega} (e^{2u} |D\log f^{3,\varepsilon}|^2)
   , \max\limits_{\Gamma_R}P\right\},\quad \ k=\frac{n}{2}  }, \\
   {
   \max \left\{\left[\frac{2k-n}{k(n+1-k)}\right]^2\max\limits_{B_R\setminus\Omega} (u^{\frac{2k}{2k-n}} |D\log f^{2,\varepsilon}|^2)
   , \max\limits_{\Gamma_R}P \right\},\quad \ k>\frac{n}{2}  },  \\
   {
   \max\left\{\left[\frac{n-2k}{k(n+1-k)}\right]^2\max\limits_{B_R\setminus\Omega} \left[(-u)^{-\frac{2k}{n-2k}} |D\log f^{1,\varepsilon}|^2\right]
   , \max\limits_{\Gamma_R}P \right\},\quad \ k<\frac{n}{2}  },
\end{array}} \right.
\end{align}
where $\Gamma_R:=\partial\left(B_R\setminus\Omega \right)$.

\end{theorem}
\begin{proof}For simplicity, we use $f$ instead of $f^{1,\varepsilon}$, $f^{2,\varepsilon}$ or $f^{3,\varepsilon}$ during the proof.
Consider the function $G=\log|Du|^2+g(u)$.
\begin{align*}
0=G_i=&\frac{|Du|^2_i}{|Du|^2}+g'u_i=\frac{2u_ku_{ki}}{u_1^2}+g'u_i\\
=&\frac{2u_{1i}}{u_1}+g'u_i.
\end{align*}
Then we have
\begin{align*}
u_{1i}=0, i\ge 2, \lambda_1=u_{11}=-\frac{g'}{2}u_1^2.
\end{align*}
In the following, we will take $g$ in three cases
\begin{align}\label{threecase1}
g(u)=
\left\{ {\begin{array}{*{20}c}
   {2u, \quad \quad \quad\quad \ k=\frac{n}{2}  }, \\
   {{\frac{2(n-k)}{(2k-n)}}\log u, \ \quad\  k>\frac{n}{2} }, \\
   {{-\frac{2(n-k)}{(n-2k)}}\log(-u), \ k<\frac{n}{2} }.  \\
\end{array}} \right.
\end{align}
In these three cases, we always have $g'>0$. This implies $\lambda_1<0$ and thus $(\lambda|1)\in \Gamma_{k}$ which is crucial during the proof.

Thus $u_{ij}$ is diagonal at $x_0$ and
\begin{align*}
0\ge F^{ii}G_{ii}=&\frac{F^{ii}|Du|^2_{ii}}{|Du|^2}-F^{ii}\left||Du|^2_i\right|^2
+g''F^{ii}u_i^2+g'F^{ii}u_{ii}\\
=&\frac{2F^{ii}u_{ki}^2+2F^{ii}u_ku_{kii}}{|Du|^2}-F^{ii}(g')^2u_i^2
+g''u_i^2+g'u_{ii}\\
=&\frac{2F^{ii}u_{ii}^2+2F^{ii}u_1u_{1ii}}{u_1^2}-(g')^2F^{11}u_1^2
+g''F^{11}u_1^2+g'F^{ii}u_{ii}\\
=&2{u_1^{-2}}\Big({S_1(\lambda)f-(k+1)S_{k+1}(\lambda)
+u_1f_1}\Big)
+\Big(g''-(g')^2\Big)S_{k-1}({\lambda|1})u_1^2+kg'f\\
=&2{u_1^{-2}}\left({S_1(\lambda)f-(k+1)S_{k+1}(\lambda)
+u_1f_1}
+\frac{1}{2}\Big(g''-(g')^2\Big)S_{k-1}({\lambda|1})u_1^4+\frac{k}{2}g'u_1^2f\right)\\
=&2{u_1^{-2}}\left({S_1(\lambda)f-(k+1)S_{k+1}(\lambda)
+u_1f_1}
+2\Big(\frac{g''}{(g')^2}-1\Big)S_{k-1}({\lambda|1})\lambda_1^2-kf\lambda_1\right),
\end{align*}
where we use $\lambda_1=-\frac{g'}{2}u_1^2$.\\
Therefore
\begin{align*}
0\ge \frac{1}{2}{u_1^{2}}F^{ii}G_{ii}=S_1(\lambda)f-(k+1)S_{k+1}(\lambda)
+u_1f_1
+2\Big(\frac{g''}{(g')^2}-1\Big)S_{k-1}({\lambda|1})\lambda_1^2-kf\lambda_1\\
=S_1(\lambda|1)f-(k-1)f\lambda_1-(k+1)S_{k+1}(\lambda)
+2\Big(\frac{g''}{(g')^2}-1\Big)S_{k-1}({\lambda|1})\lambda_1^2+u_1f_1.
\end{align*}
Sine $\lambda_1S_{k-1}(\lambda|1)+S_{k}(\lambda|1)=f$, we have $\lambda_1=\frac{f}{S_{k-1}(\lambda|1)}-
\frac{S_{k}(\lambda|1)}{S_{k-1}(\lambda|1)}
$.
We first manipulate the term $-(k-1)f\lambda_1$.
\begin{align}
-(k-1)f\lambda_1
=&-(k-1)f\Big(\frac{f}{S_{k-1}(\lambda|1)}-
\frac{S_{k}(\lambda|1)}{S_{k-1}(\lambda|1)}\Big)\notag\\
=&-(k-1)\frac{f^2}{S_{k-1}(\lambda|1)}
+(k-1)f\frac{S_{k}(\lambda|1)}{S_{k-1}(\lambda|1)}.
\end{align}
Next we manipulate $-(k+1)S_{k+1}(\lambda)$.
\begin{align}
-(k+1)S_{k+1}(\lambda)=&-(k+1)\lambda_1S_{k}(\lambda|1)
-(k+1)S_{k+1}(\lambda|1)\notag\\
=&-(k+1)S_{k}(\lambda|1)\Big(\frac{f}{S_{k-1}(\lambda|1)}-
\frac{S_{k}(\lambda|1)}{S_{k-1}(\lambda|1)}\Big)
-(k+1)S_{k+1}(\lambda|1)\notag\\
=&-(k+1)f\frac{S_{k}(\lambda|1)}{S_{k-1}(\lambda|1)}
+(k+1)\frac{S_{k}^2(\lambda|1)}{S_{k-1}(\lambda|1)}-(k+1)S_{k+1}(\lambda|1)
\end{align}

At last, we manipulate the trouble term $2\Big(\frac{g''}{(g')^2}-1\Big)S_{k-1}({\lambda|1})\lambda_1^2$.
\begin{align}
&2\Big(\frac{g''}{(g')^2}-1\Big)S_{k-1}({\lambda|1})\lambda_1^2\notag\\
&=2\Big(\frac{g''}{(g')^2}-1\Big)S_{k-1}({\lambda|1})\Big(\frac{f}{S_{k-1}(\lambda|1)}-
\frac{S_{k}(\lambda|1)}{S_{k-1}(\lambda|1)}\Big)^2\notag\\
=&2\Big(\frac{g''}{(g')^2}-1\Big)\frac{f^2}{S_{k-1}(\lambda|1)}
-4\Big(\frac{g''}{(g')^2}-1\Big)f\frac{S_{k}(\lambda|1)}
{S_{k-1}(\lambda|1)}+2\Big(\frac{g''}{(g')^2}-1\Big)\frac{S_{k}^2
(\lambda|1)}{S_{k-1}(\lambda|1)}.
\end{align}
Substitute the above three equality into the original terms, we have
\begin{align}\label{Final1}
0\ge& \frac{1}{2}{u_1^{2}}F^{ii}G_{ii}\notag\\
=&\Big(2\frac{g''}{(g')^2}+k-1\Big)\frac{S_{k}^2
(\lambda|1)}{S_{k-1}(\lambda|1)}-(k+1)S_{k+1}(\lambda|1)+
2\Big(1-\frac{2g''}{(g')^2}\Big)f\frac{S_{k}
(\lambda|1)}{S_{k-1}(\lambda|1)}\notag\\
&+\Big(2\frac{g''}{(g')^2}-k-1\Big)\frac{f^2}{S_{k-1}(\lambda|1)}
+S_1(\lambda|1)f+u_1f_1\notag\\
\ge
&\Big(2\frac{g''}{(g')^2}+k-1-\frac{k(n-k-1)}{n-k}\Big)\frac{S_{k}^2
(\lambda|1)}{S_{k-1}(\lambda|1)}+
2\Big(1-\frac{2g''}{(g')^2}\Big)f\frac{S_{k}
(\lambda|1)}{S_{k-1}(\lambda|1)}\notag\\
&+\Big(2\frac{g''}{(g')^2}-k-1\Big)\frac{f^2}{S_{k-1}(\lambda|1)}
+S_1(\lambda|1)f+u_1f_1,
\end{align}
where in the last inequality we use the Maclaurin inequality:
\[
\frac{S_{k+1}(\lambda|1)/C_{n-1}^{k+1}}{S_{k}(\lambda|1)/C_{n-1}^{k}}\le
\frac{S_{k}(\lambda|1)/C_{n-1}^{k}}{S_{k-1}(\lambda|1)/C_{n-1}^{k-1}}.
\]

 \textbf{ Case1: $k<\frac{n}{2}$}, \\
Since the foundamental solution is $-|x|^{2-\frac{n}{k}}$ and its gradient is $\sim |x|^{1-\frac{n}{k}}$.
We take $g(u)=a\log(-u)$, where $a=-\frac{2(n-k)}{n-2k}<0$.
$
2\frac{g''}{(g')^2}=-\frac{2}{a}=\frac{n-2k}{n-k}$. Substituting it into  \eqref{Final}, we have

\begin{align}\label{Final}
0\ge& \frac{1}{2}{u_1^{2}}F^{ii}G_{ii}\ge\Big(2\frac{g''}{(g')^2}+k-1-\frac{k(n-k-1)}{n-k}\Big)\frac{S_{k}^2
(\lambda|1)}{S_{k-1}(\lambda|1)}\notag\\
&+
2\Big(1-\frac{2g''}{(g')^2}\Big)f\frac{S_{k}
(\lambda|1)}{S_{k-1}(\lambda|1)}+\Big(2\frac{g''}{(g')^2}-k-1\Big)\frac{f^2}{S_{k-1}(\lambda|1)}
+S_1(\lambda|1)f+u_1f_1\notag\\
=&\frac{2k}{n-k}f\frac{S_{k}
(\lambda|1)}{S_{k-1}(\lambda|1)}-k\frac{n+1-k}
{n-k}\frac{f^2}{S_{k-1}(\lambda|1)}
+S_1(\lambda|1)f+u_1f_1\notag\\
\ge&f\Big(\frac{k}{n-1}S_1(\lambda|1)-k\frac{n+1-k}
{n-k}\lambda_1  \Big)+u_1f_1 (*)\notag\\
\ge & \frac{a}{2}k\frac{n+1-k}
{n-k}=\frac{k(n+1-k)}
{n-2k} f\frac{u_1^2}{-u} +u_1f_1,
\end{align}
where in (*), we use the same argument as the $k>\frac{n}{2}$ case.\\
From this, we have
$
u_1\le \frac{n-2k}{k(n+1-k)}(-u)|D\log f|.
$
Then we have
\begin{align*}
|Du|^2(-u)^a\le u_1^2(-u)^a\le\Big(\frac{n-2k}{k(n+1-k)}\Big)^2(-u)^{\frac{2k}{2k-n}}|D\log f|^2.
\end{align*}
\textbf{Case2: $k>\frac{n}{2}$.}\\
 We take $g=a\log u$, where $a=\frac{2(n-k)}{2k-n}$ (since the foudament solution is $|x|^{2-\frac{n}{k}}$ ). By \eqref{Final}, we have
\begin{align*}
0\ge& \frac{1}{2}{u_1^{2}}F^{ii}G_{ii}\notag\\
=&2\left[1-\frac{2g''}{(g')^2}\right]f\frac{S_{k}
(\lambda|1)}{S_{k-1}(\lambda|1)}
+\left[2\frac{g''}{(g')^2}-k-1\right]\frac{f^2}{S_{k-1}(\lambda|1)}
+S_1(\lambda|1)f+u_1f_1\\
=&2(1+\frac{2}{a})f\frac{S_{k}
(\lambda|1)}{S_{k-1}(\lambda|1)}-(1+\frac{2}{a}+k)\frac{f^2}{S_{k-1}(\lambda|1)}
+S_1(\lambda|1)f+u_1f_1\\
=&\frac{2k}{n-k}f\frac{S_{k}
(\lambda|1)}{S_{k-1}(\lambda|1)}-k\frac{n+1-k}{n-k}\frac{f^2}{S_{k-1}(\lambda|1)}
+S_1(\lambda|1)f+u_1f_1\\
=&\frac{2k}{n-k}f\frac{S_{k}
(\lambda|1)}{S_{k-1}(\lambda|1)}-k\frac{n+1-k}
{n-k}f\left[\lambda_1+\frac{S_{k}
(\lambda|1)}{S_{k-1}(\lambda|1)}\right]
+S_1(\lambda|1)f+u_1f_1\\
=&f\left[S_1(\lambda|1)+\frac{k(k+1-n)}{n-k}\frac{S_{k}
(\lambda|1)}{S_{k-1}(\lambda|1)}-k\frac{n+1-k}
{n-k}\lambda_1\right]+u_1f_1\\
\ge& f\left[\frac{k}{n-1}S_1(\lambda|1)-k\frac{n+1-k}
{n-k}\lambda_1   \right]+u_1f_1\\
\ge& c_{n,k}f\frac{u_1^2}{u} +u_1f_1,
\end{align*}
where $c_{n,k}=\frac{a}{2}k\frac{n+1-k}
{n-k}=\frac{k(n+1-k)}
{2k-n}>0$.

This gives
\begin{align*}
u_1\le \frac{2k-n}{k(n+1-k)}u|D\log f|.
\end{align*}
Therefore we have
\begin{align}\label{case2}
|Du|^2u^a=u_1^2u^a\le \Big(\frac{2k-n}{k(n+1-k)}\Big)^2u^{\frac{2k}{2k-n}}|D\log f|^2.
\end{align}

\textbf{Case3: $k=\frac{n}{2}$},\\
In this case, we must choose $g(u)=2u$ (it is uniquely determined by the foundamental solution  $\log |x|$), then from \eqref{Final1}, we have
\begin{align*}
0\ge& \frac{1}{2}{u_1^{2}}F^{ii}G_{ii}\ge 2f\frac{S_{k}(\lambda|1)}{S_{k-1}(\lambda|1)}
-(\frac{n}{2}+1)\frac{f^2}{S_{k-1}(\lambda|1)}
+S_1(\lambda|1)f+u_1f_1\\
=&f\Big(S_1(\lambda|1)+2\frac{S_{k}(\lambda|1)}{S_{k-1}(\lambda|1)}
-(\frac{n}{2}+1)\frac{f}{S_{k-1}(\lambda|1)}\Big)+u_1f_1\\
=&f\Big(S_1(\lambda|1)+2\frac{S_{k}(\lambda|1)}{S_{k-1}(\lambda|1)}
-(\frac{n}{2}+1)\Big(\lambda_1+
\frac{S_{k}(\lambda|1)}{S_{k-1}(\lambda|1)}\Big)\Big)+u_1f_1\\
=&f\Big(-(\frac{n}{2}+1)\lambda_1+S_1(\lambda|1)
-(\frac{n}{2}-1)
\frac{S_{k}(\lambda|1)}{S_{k-1}(\lambda|1)}\Big)+u_1f_1\\
\ge& f\Big((\frac{n}{2}+1)\lambda_1+\frac{n-2}{2(n-1)}S_1(\lambda|1)
\Big)+u_1f_1\\
\ge& f(\frac{n}{2}+1)u_1^2+u_1f_1.
\end{align*}
This implies
$
u_1\le -(\log f)_1\le |D\log f|.
$
Thus we have
\begin{align*}
|Du|^2e^{2u}\le u_1^2 e^{2u}\le e^{2u} |D\log f|^2.
\end{align*}

\end{proof}
\subsubsection\textbf{\textbf{{ Boundary gradient estimates}}}

We always assume $R>>100(1+R_0)$. To prove the boundary gradient estimates, we will construct upper barriers on $\partial \Omega$ and $\partial B_R$ respectively.

\emph{\textbf{Case1: $k<\frac{n}{2}$}}\\
Let  $h\in C^{\infty}(\overline\Omega_{R_0})$ be the unique solution of
\begin{align*}
\left\{\begin{aligned}
\Delta h=&0,\ \text{in}\  \Omega_{R_0},\\
h=&-1, \ \text{on}\ \partial \Omega,\\
h=&-r_0^{\frac{n-2k}{k}}{R_0}^{-\frac{n-2k}{k}},
\ \text{on}\ \partial B_{R_0}.
\end{aligned}
\right.
\end{align*}
By maximum principle, $\underline{u}^{1,\varepsilon}\le u\le h^{R}$ in $\overline\Omega_{R_0}$. Then for any $x\in\partial \Omega$
\begin{align*}
0>-\tau_0=\tau_0\Phi_{\nu}(x)={\underline{u}^{1,\varepsilon}}_{\nu}(x)\ge u_{\nu}(x)\ge h_{\nu}(x),
\end{align*}
where  $\nu$ is the outward normal  of $\partial \Omega_R$ (inward normal of $\p\Omega$). Then
\begin{align}\label{0714gradient1.1}
0<\tau_0\le \tau_0\max\limits_{\partial\Omega}|Du|=
\max\limits_{\partial\Omega}(-u_{\nu})\le \max_{\partial\Omega}|h_{\nu}|\le C.
\end{align}
This proves that $P$ is uniformly bounded on $\partial \Omega$.

Next we show $P$ is uniformly bounded on $\partial B_R$.
Indeed, we consider
\begin{align}\label{0727rho1}
\rho^{\varepsilon,R}=-a^{\varepsilon,R}r_0^{\frac{n-2k}{k}}|x|^{-\frac{n-2k}{k}}+a^{\varepsilon,R}-1,
\end{align}
where $a^{\varepsilon,R}$ is defined as
follows
\begin{align}
	a^{\varepsilon,R}=\left(1-\left(\frac{r_0}{R_0}\right)^{\frac{n-2k}{k}}\right)^{-1}\left(\Big(\frac{R_0^2+\varepsilon^2}{R^2+\varepsilon^2}\Big)^{\frac{n-2k}{2k}}+1\right)> 1.
	\end{align}
Then we have $\rho^{\varepsilon,R}=u$ on $\partial B_R$ and $\underline{u}^{1,\varepsilon}\le u\le \rho^{\varepsilon,R}$ in $\Omega_R$.
Then for any $x\in\p B_R$
\begin{align}\label{457141}
CR^{-\frac{n-k}{k}}\ge{\underline{u}^{1,\varepsilon}}_{\nu}(x)\ge u_{\nu}(x)\ge \rho^{\varepsilon, R}_{\nu}(x)=a^{\varepsilon,R}\frac{n-2k}{k}r_0^{\frac{n-2k}{k}}R^{-\frac{n-k}{k}}\ge cR^{-\frac{n-k}{k}},
\end{align}
where $C$ and $c$ are uniformly positive constants.
Thus
\begin{align}\label{0714gradient1.2}
	cR^{-\frac{n-k}{k}}\le \max\limits_{\p B_R}|Du|=\max\limits_{\p B_R} u_{\nu}\le CR^{-\frac{n-k}{k}}.
\end{align}
Combing \eqref{0714gradient1.1} with \eqref{0714gradient1.2}, we have
\begin{align}\label{0714gradient1}
	c|x|^{1-\frac{n}{k}}\le |Du|\le C|x|^{1-\frac{n}{k}} \ \text{on}\ \partial\Omega_R
\end{align}
This implies $P$ is uniformly bounded on $\partial B_R$.

 In conclusion, when $k<\frac{n}{2}$, $P$ is uniformly bounded in $\overline\Omega_R$.

 \emph{\textbf{Case 2: $k>\frac{n}{2}$}}\\
 Let  $h\in C^{\infty}(\overline\Omega_{R_0})$ be the unique solution of
 \begin{align*}
 	\left\{\begin{aligned}
 	\Delta h=&0,\ \text{in}\  \Omega_{R_0},\\
 		h=&1, \ \text{on}\ \partial \Omega,\\
 		h=&R_0^{\frac{2k-n}{k}}-r_0^{\frac{2k-n}{k}}+1,
 		\ \text{on}\ \partial B_{R_0}.
 	\end{aligned}
 	\right.
 \end{align*}
 By maximum principle, $\underline{u}^{2,\varepsilon}\le u\le h$ in $\overline\Omega_{R_0}$.
Then
\begin{align}\label{0714gradient2.1}
	0<c\le |Du|\le C, \ \text{on}\ \partial\Omega
	\end{align}
 Thus we have $P$ is uniformly bounded on $\partial\Omega$ .

 We construct the upper barrier of $u$ in $\Omega_R$ as follows
 \begin{align}\label{0727rho2}
 	\rho^{\varepsilon,R}=a^{\varepsilon,R}\left(|x|^{\frac{2k-n}{k}}-r_0^{\frac{2k-n}{k}}\right)+1,
 \end{align}
 where $a^{\varepsilon,R}$ is defined by
 \begin{align}\label{a0715}
 	a^{\varepsilon,R}=\left(R^{\frac{2k-n}{k}}-r_0^{\frac{2k-n}{k}}\right)^{-1}\left((R^2+\varepsilon^2)^{1-\frac{n}{2k}}-(R_0^2+\varepsilon^2)^{1-\frac{n}{2k}}\right)> a_0>0,
 \end{align}
 where $a_0>0$ is independent of $\varepsilon$ and $R$. Then we have
 $\rho^{\varepsilon,R}=u$ on $\partial B_R$ and $\underline{u}^{2,\varepsilon}\le u\le \rho^{\varepsilon,R}$ in $\Omega_R$.
Thus
 \begin{align}\label{0714gradient2.2}
 	cR^{-\frac{n-k}{k}}\le \max\limits_{\p B_R}|Du|=\max\limits_{\p B_R} u_{\nu}\le CR^{-\frac{n-k}{k}}.
 \end{align}
Combing \eqref{0714gradient2.1} with \eqref{0714gradient2.2}, we have
\begin{align}\label{0714gradient2}
	c|x|^{1-\frac{n}{k}}\le |Du|\le C|x|^{1-\frac{n}{k}} \ \text{on}\ \partial\Omega_R
\end{align}
 This implies $P$ is uniformly bounded on $\partial B_R$.
 In conclusion, when $k>\frac{n}{2}$, $P$ is uniformly bounded in $\overline\Omega_R$.

  \emph{\textbf{Case 3: $k=\frac{n}{2}$}}\\
 Let  $h\in C^{\infty}(\overline\Omega_{R_0})$ be the unique solution of
 \begin{align*}
 	\left\{\begin{aligned}
 		\Delta h=&0,\ \text{in}\  \Omega_{R_0},\\
 		h=&0, \ \text{on}\ \partial \Omega,\\
 		h=&\log R_0-\log r_0,
 		\ \text{on}\ \partial B_{R_0}.
 	\end{aligned}
 	\right.
 \end{align*}
 By maximum principle, $\underline{u}^{3,\varepsilon}\le u\le h$ in $\overline\Omega_{R_0}$.
 Then
 \begin{align}\label{0714gradient3.1}
 	0<c\le |Du|\le C, \ \text{on}\ \partial\Omega
 \end{align}
 Thus  $P$ is uniformly bounded on $\partial\Omega$ .

 We construct the upper barrier of $u$  in $\Omega_R$ as follows
 \begin{align*}
 	\rho^{\varepsilon,R}=a^{\varepsilon, R}\left(\log|x|-\log r_0\right),
 \end{align*}
 where $a^{\varepsilon, R}$ is defined by
 \begin{align}
 	a^{\varepsilon, R}=\left(\log R-\log r_0\right)^{-1}\frac{1}{2}\log\frac{ R^2+\varepsilon^2}{ R_0^2+\varepsilon^2}> a_0>0,
 \end{align}
 where $a_0>0$ is independent of $\varepsilon$ and $R$. Then we have
 $\rho^{\varepsilon,R}=u$ on $\partial B_R$ and $\underline{u}^{3,\varepsilon}\le u\le \rho^{\varepsilon,R}$ in $\Omega_R$.
 Thus
 \begin{align}\label{0714gradient3.2}
 	cR^{-1}\le \max\limits_{\p B_R}|Du|=\max\limits_{\p B_R} u_{\nu}\le CR^{-1}.
 \end{align}
Combing \eqref{0714gradient3.1} with \eqref{0714gradient3.2}, we have
\begin{align}\label{0714gradient3}
	c|x|^{-1}\le |Du|\le C|x|^{-1} \ \text{on}\ \partial\Omega_R
	\end{align}
 This implies $P$ is uniformly bounded on $\partial \Omega_R$.
 In conclusion, when $k=\frac{n}{2}$, $P$ is uniformly bounded in $\overline\Omega_R$.
 \subsubsection{\textbf{Positive lower bound of $|Du|$}.}
 \begin{lemma}
 	Let $u$ be the $k$-convex solution of the approximating equation \eqref{case1EquaAppr}, \eqref{case2EquaAppr} or \eqref{case3EquaAppr}.
 For sufficiently large $R$ and sufficiently small $\varepsilon$,	there exists a uniform constant $c_0$ such that for any $x\in \overline\Omega_R$
 	\begin{align}\label{0714gl1}
 		x\cdot Du(x)\ge
 		\left\{\begin{aligned} &c_0 |x|^{2-\frac{n}{k}}, \ \text{if} \ k<\frac{n}{2}  \ \text{or}\  k>\frac{n}{2},\\
 			&c_0,\qquad \text{if} \ k=\frac{n}{2}.
 			\end{aligned}
 			\right.
 	\end{align}
 	In particular,
 	\begin{align}\label{0715gl1}
 		| Du(x)|\ge
 		\left\{\begin{aligned} &c_0 |x|^{1-\frac{n}{k}}, \ \text{if} \ k<\frac{n}{2}  \ \text{or}\  k>\frac{n}{2},\\
 			&c_0|x|^{-1},\quad \text{if} \ k=\frac{n}{2}.
 		\end{aligned}
 		\right.
 	\end{align}
 \end{lemma}

 \begin{proof}
 	\emph{Case 1: $k< \frac{n}{2}$}
 	
 	We consider the function $H:=x\cdot Du(x)+b_1 u$. Recall $F^{ij}=\frac{\partial }{\partial u_{ij}}(\log S_k(D^2 u))$.
 	
 	Direct manipulation gives
 	\begin{align}
 		F^{ij}H_{ij}=(2+b_1)k+x_m\tilde f_m,
 	\end{align}
 	Recall $\tilde f=\log f$, we have
 	\begin{align}
 		x_m\tilde f_m=-(n+2)|x|^2(|x|^2+\varepsilon^2)^{-1}.
 	\end{align}
 	Then if $b_1<k^{-1}$, we have
 	\begin{align}\label{0714gl2}
 		F^{ij}H_{ij}=&(|x|^2+\varepsilon^2)^{-1}\left(\Big((2+b_1)k-(n+2)\Big)|x|^2+(2+b_1)k\varepsilon^2\right)<0,
 	\end{align}
 By maximum principle, $H\ge \min_{\partial \Omega_R} H $. By choosing $b_1$ sufficiently small, we can prove $\min_{\partial \Omega_R} H>0$.

 Indeed, for any  $x\in \partial \Omega$, since $\Omega$ is convex, we have
 \begin{align}
 	H(x)=&x\cdot Du(x)-b_1
 	=(x\cdot \nu(x))|Du(x)|-b_1\notag\\
 	\ge& \min_{\partial\Omega}(x\cdot \nu(x))c-b_1>0,
 	\end{align}
 where the last term is positive if we choose $b_1<\min_{\partial\Omega}(x\cdot \nu(x))c$ and $\nu(x)$ is the outward unit normal vector of $\Omega$ at $x\in\p\Omega$.

 Indeed, for any  $x\in \partial B_R$, we have
 \begin{align}
 	H(x)=&x\cdot Du(x)-b_1(R_0^2+\varepsilon^2)^{\frac{n}{k}-2}(R^2+\varepsilon^2)^{2-\frac{n}{k}}\notag\\
 	=&Ru_{\nu}-b_1(R_0^2+\varepsilon^2)^{\frac{n}{k}-2}(R^2+\varepsilon^2)^{\frac{2k-n}{2k}}\notag\\
 	\ge& cR^{\frac{2k-n}{k}}-C_1b_1R^{\frac{2k-n}{k}}\notag\\
 	>&0,
 \end{align}
where we require $b_1<\frac{c}{2C_1}$.
In conclusion, we prove $H>0$ in $\overline \Omega_R$ and thus we prove \eqref{0714gl1}.
 	
 	\emph{Case 2: $k> \frac{n}{2}$}
 	
 	 We consider the function $H:=x\cdot Du(x)-b_2 (u-1)-a_2$, where $ b_2$ and $a_2$ are  positive constants to be determined later with $a_2<\frac{b_2}{2}$

For any $x\in {\partial \Omega}$, since $u(x)=1$, we have $H(x)=x\cdot Du(x)-a_2\ge c(x\cdot\nu(x))-a_2>0$ if $a_2<\min_{\p\Omega}(x\cdot\nu(x))c$ is small enough.

For any $x\in {\partial B_R} $, recall the upper barrier  $\rho^{\varepsilon,R}$ in \eqref{0727rho2} and  $a^{\varepsilon,R}$ in \eqref{a0715}, we have
\begin{align*}
	H(x)=&Ru_{\nu}-b_2\left(\Big(R^2+\varepsilon^2\Big)^{\frac{2k-n}{2k}}-\Big(R_0^2+\varepsilon^2\Big)^{\frac{2k-n}{2k}}-1\right)-a_2\\
\ge& R\rho_{\nu}^{\varepsilon,R}-b_2\Big(R^2+\varepsilon^2\Big)^{\frac{2k-n}{2k}}\\
	=&\frac{2k-n}{k}a^{\varepsilon,R}R^{\frac{2k-n}{k}}-b_2\Big(R^2+\varepsilon^2\Big)^{\frac{2k-n}{2k}}\\
	=&R^{\frac{2k-n}{k}}\left(\frac{2k-n}{k}a^{\varepsilon,R}-b_2\Big(1+\varepsilon^2R^{-2}\Big)^{\frac{2k-n}{2k}}\right).
	\end{align*}
If we take $b_2=\frac{2k-n}{k}-\frac{1}{2k}>0$ (since $2k-n\ge 1$), the above is positive since $a^{\varepsilon,R}$ is close to $1$ for $R$ sufficiently large.
For such $b_2$, we have
 \begin{align}\label{0714gl3}
 	F^{ij}H_{ij}=&(|x|^2+\varepsilon^2)^{-1}\left(\Big((2-b_2)k-(n+2)\Big)|x|^2+(2-b_2)k\varepsilon^2\right)
 	\notag\\
 	\le& (|x|^2+\varepsilon^2)^{-1}\Big(-|x|^2+(n+1)\varepsilon^2\Big)\notag\\
 	<&0,
 \end{align}
where we assume $\varepsilon$ small enough (note that $|x|\ge r_0$ for $x\in\Omega^c$).

By maximum principle, we have
$H>\min_{\partial \Omega_R} H>0$. Thus we get for any $x\in \overline\Omega_R$,

\begin{align*}
	x\cdot Du(x)\ge& b_2(u-1)+a_2\notag\\
	\ge& \frac{a_2}{2}u
	\ge\frac{a_2}{4(1+C)} |x|^{\frac{2k-n}{k}},
	\end{align*}
where we use $u\ge \max\{|x|^{\frac{2k-n}{k}}-C, 1\}$.

\emph{Case 3: $k=\frac{n}{2}$}

We consider $H=x\cdot Du(x)-b_3$ which is positive on the boundary of $\overline \Omega_R$ if we take $b_3$ small enough. Since $F^{ij}H_{ij}=(|x|^2+\varepsilon^2)^{-1}(-2|x|^2+n\varepsilon^2)<0$ for $\varepsilon$ small enough, we have $H=x\cdot Du(x)-b_3>0$ in $\overline \Omega_R$ and we can get the desired estimate.

%
 	\end{proof}

\subsection{Second order estimates}
We will prove the second order estimate of the approximating equations.

\subsubsection{\textbf{The global second order estimate can be reduced to the boundary second order estimate} }

\begin{theorem}
Let $u$ be the $k$-convex solution of \eqref{case1EquaAppr} or \eqref{case2EquaAppr} or \eqref{case3EquaAppr} and consider
$
\tilde G=u_{\xi\xi}\varphi(P)h(u)
$,
then we have
\begin{align}\label{Secondorderestimate}
\max\limits_{B_R\setminus\Omega} \tilde G\le C+\max\limits_{\Gamma_R} \tilde G.
\end{align}
where $\Gamma_R:=\partial\left(B_R\setminus\Omega \right)$, $\varphi(t)$ and $h$ are defined by
\begin{align}
	\varphi(t)=\left\{
	\begin{aligned}
		(M-t)^{-\tau}, k<n,\\
		1,           \qquad k=n,
		\end{aligned}
	\right.
	\end{align}
where $M:=2\max P+1$ and $\tau$ is a uniform positive constant determined in \eqref{0727tau1} (if $k\neq\frac{n}{2}$) and \eqref{0727tau2} (if $k=\frac{n}{2}$). 
$h$ is defined by 
\begin{align}
	h(u)=
	\left\{ {\begin{array}{*{20}c}
			{e^{2u}, \quad  \quad\ k=\frac{n}{2}  }, \notag\\
			{ u^{\frac{n}{2k-n}}, \ \ \quad k>\frac{n}{2} }, \notag\\
			{(-u)^{-\frac{n}{n-2k}}, \ k<\frac{n}{2} }. \notag \\
	\end{array}} \right.
\end{align}
\end{theorem}
\begin{proof}
	For simplicity, we write $f$ instead of $f^{1,\varepsilon}$, $f^{2,\varepsilon}$ or $f^{3,\varepsilon}$ during the proof.\\
We rewrite the equation as
\begin{align*}
F(D^2u)=\log {S_k(D^2 u)}=\tilde f,
\end{align*}
where $\tilde f=\log f$ satisfying
\begin{align}\label{f0714}
	|D\tilde f|^2+|D^2\tilde f|\le C|x|^{-2}.
	\end{align}
Now we are ready to prove the second order estimate, we first recall
\begin{align}
P=|Du|^2 \tilde g(u):=\left\{ {\begin{array}{*{20}c}
   {|Du|^2e^{2u}, \quad \quad  \ k=\frac{n}{2}  }, \notag\\
   {|Du|^2 u^{\frac{2(n-k)}{2k-n}} , \ \quad\  k>\frac{n}{2} },\notag \\
   {|Du|^2 (-u)^{-\frac{2(n-k)}{n-2k}}, \ k<\frac{n}{2} }.  \notag\\
\end{array}} \right.
\end{align}

Direct manipulation shows that
\begin{align}
	(\log h)'=&
	\left\{ {\begin{array}{*{20}c}
			{2, \quad  \quad\ k=\frac{n}{2}  }, \notag\\
			{ \frac{n}{2k-n}u^{-1}, \ \ \quad k>\frac{n}{2} }, \notag\\
			{\frac{n}{n-2k}(-u)^{-1}, \ k<\frac{n}{2} }, \notag \\
	\end{array}} \right.\\
h((\log h)')^{-1}\le& C |x|^2\label{h0714},
	\end{align}
moreover,
\begin{align}\label{ghrelation}
	\tilde g=c_{n,k}h(\log h)',
	\end{align}
where $c_{n,k}=2$ when $k=\frac{n}{2}$ and $c_{n,k}=\frac{n}{|n-2k|}$ when $k\neq\frac{n}{2}$.

Assume $\tilde G$ attains its maximum at $x_0\in \Omega_R$ along the direction $\xi_0$. We choose the coordinate at $x_0\in \Omega_R$ such that $D^2 u(x_0)=\{\lambda_i\delta_{ij}\}$. Then one can check $\xi_0=(1,0\cdots,0)$.
Then
$
G=\log{u_{11}}+\log{\varphi(P)}+\log{h(u)}$
attains its maximum at $x_0$.\\
\emph{ Our goal is to prove the uniform upper  bound of $\lambda_1 h$.}

 All the calculations are at $x_0$, fisrtly, we have
\begin{align*}
0=G_i=&\frac{u_{11i}}{u_{11}}+\frac{\varphi_i}{\varphi}+\frac{h_i}{h}\\
=&\frac{u_{11i}}{u_{11}}+\frac{\varphi'}{\varphi}P_i+\frac{h'}{h}u_i.
\end{align*}
From the above, we have
\begin{align}
\frac{u_{11i}}{u_{11}}=&-\frac{\varphi'}{\varphi}P_i-\frac{h'}{h}u_i,\label{1}\\
\frac{h'}{h}u_i=&-\frac{u_{11i}}{u_{11}}-\frac{\varphi'}{\varphi}P_i.\label{2}
\end{align}

Differentiating $G$ twice, we have
\begin{align}
G_{ii}=\frac{u_{11ii}}{u_{11}}-\left(\frac{u_{11i}}{u_{11}}\right)^2
+\frac{\varphi'}{\varphi}P_{ii}+\left(\frac{\varphi''}{\varphi}-\left(\frac{\varphi'}{\varphi}  \right)^2\right)|P_i|^2
+\frac{h''}{h}u_{ii}+\left(\frac{h''}{h}-\left(\frac{h'}{h}  \right)^2\right)|u_i|^2.
\end{align}
Then we have
\begin{align}
0\ge F^{ii}G_{ii}=&\frac{F^{ii}u_{11ii}}{u_{11}}-
{F^{ii}}\left(\frac{u_{11i}}{u_{11}}\right)^2
+\frac{\varphi'}{\varphi}F^{ii}P_{ii}
+\left(\frac{\varphi''}{\varphi}-\left(\frac{\varphi'}{\varphi}  \right)^2\right)F^{ii}|P_i|^2
+\frac{h''}{h}F^{ii}u_{ii}+\left(\frac{h''}{h}-\left(\frac{h'}{h}  \right)^2\right)F^{ii}|u_i|^2\notag\\
=&\frac{(\tilde f)_{11}-F^{ij,rs}u_{ij1}u_{rs1}}{u_{11}}-
\sum_{i=2}^n{F^{ii}}\left(\frac{u_{11i}}{u_{11}}\right)^2
-F^{11}\left(\frac{\varphi'}{\varphi}P_1+\frac{h'}{h}u_1\right)^2
+\left(\frac{\varphi''}{\varphi}-\left(\frac{\varphi'}{\varphi}  \right)^2\right)F^{ii}|P_i|^2\notag\\
&+\frac{\varphi'}{\varphi}F^{ii}P_{ii}+\left(\frac{h''}{h}-\left(\frac{h'}{h}  \right)^2\right)F^{ii}|u_i|^2+\frac{h'}{h}F^{ii}u_{ii}\notag\\
\ge& 2\lambda_1^{-1}f^{-1}\sum_{i=2}^n S_{k-2}(\lambda|1i)|u_{11i}|^2-
\sum_{i=2}^n{F^{ii}}\left(\frac{u_{11i}}{u_{11}}\right)^2
\notag\\
&+\left(\frac{\varphi''}{\varphi}-\left(\frac{\varphi'}{\varphi}  \right)^2\right)\sum_{i=2}^nF^{ii}|P_i|^2+\left(\frac{h''}{h}-\left(\frac{h'}{h}  \right)^2\right)\sum\limits_{i=2}^n F^{ii}|u_i|^2\notag\\
&+\left(\frac{\varphi''}{\varphi}-3\left(\frac{\varphi'}{\varphi}  \right)^2\right)F^{11}|P_1|^2+\left(\frac{h''}{h}-3\left(\frac{h'}{h}  \right)^2\right)F^{11}|u_1|^2\notag\\
&+\frac{\varphi'}{\varphi}F^{ii}P_{ii}+\frac{kh'}{h}-\lambda_1^{-1}(\tilde f)_{11},\label{main}
\end{align}
where we use the concavity property of  $\log S_k$.

We fist deal with the easy case: $k=n$. Note that in this case $\varphi=1$, $h=u$ and $F^{ii}=\lambda_i^{-1}$. From
\eqref{main}, we have
\begin{align}
	0\ge F^{ii}G_{ii}\ge {n}{u^{-1}}-2\lambda_1^{-1}u^{-2}{u_1^2}-\lambda_1^{-1}(\tilde f)_{11}.
	\end{align}
Multiply $\lambda_1u^2$ in the above inequality,  we have
\begin{align}
	\lambda_1 u\le 2u_1^2+u^2(\tilde f)_{11}\le C,
\end{align}
where we use $|Du|\le C$, $u\le C|x|$  and $|D^2\tilde f|\le C|x|^{-2}$ (see \eqref{f0714}). Thus we finish the proof when $k=n$.

\textbf{In the remaining proof, we always assume $k<n$}.

We manipulate $P_{ii}$ directly.
\begin{align*}
P_i=&2\tilde{g}\sum_{l=1}^n u_l u_{li}+|Du|^2\tilde g'u_i,\\
P_{ii}=&2\tilde g\sum_{l=1}^nu_l u_{lii}+2\tilde gu_{ii}^2+4\tilde g'u_{ii}u_i^2+|Du|^2\tilde g''u_i^2+|Du|^2\tilde g'u_{ii}
\end{align*}
Thus we have
\begin{align}
F^{ii}P_{ii}=&2\tilde g\sum_{i=1}^nF^{ii}u_{ii}^2+4g'\sum_{i=1}^nF^{ii}u_{ii}u_i^2+|Du|^2\tilde {g}''\sum_{i=1}^nF^{ii}u_i^2+k\tilde g'|Du|^2+2\tilde g\sum_{i=1}^nu_i\tilde{f}_i\notag\\
\ge& \tilde g\sum_{i=1}^nF^{ii}u_{ii}^2-4\tilde g^{-1}(\tilde g')^2\sum_{i=1}^nF^{ii}u_{i}^4+|Du|^2|\tilde g''|\sum_{i=1}^nF^{ii}u_i^2-2\tilde g|Du||D\tilde f|\notag\\
\ge& \tilde g\sum_{i=1}^nF^{ii}u_{ii}^2-C((\log h)')^2|Du|^2\tilde g\sum_{i=1}^nF^{ii}u_i^2-2c_{n,k}h(\log h)' |Du||D\tilde f|\notag\\
\ge& c_{n,k}h(\log h)'\sum_{i=1}^nF^{ii}u_{ii}^2-C((\log h)')^2P\sum_{i=1}^nF^{ii}u_i^2-CP^{\frac{1}{2}} (\log h)'\label{FP0},
\end{align}
where we have used $P=|Du|^2\tilde g$, $\tilde g=c_{n,k}h(\log h)'$ and $\tilde g^{\frac{1}{2}}|D\tilde f|\le C(\log h)'$
 which follows from $|D\tilde f|\le C|x|^{-1}$ (see \eqref{f0714}) and $h((\log h)')^{-1}\le C|x|^2$(see \eqref{h0714}).

Similar as  Chou-Wang [CPAM, 2001], we now divide two cases to obtain the upper bound of $\lambda_1 h(u)$.

\emph{\textbf{Case1: $\lambda_k\ge \delta\lambda_1$}},\\
Since $\lambda_k\ge \delta\lambda_1$, there exists a constant $\theta$ such that $S_{k-1}(\lambda|k)\ge\theta S_{k-1}(\lambda) $, we have
\begin{align*}
\sum_{i=1}^n F^{ii}u_{ii}^2\ge& F^{kk}u_{kk}^2\ge \theta f^{-1}S_{k-1}(\lambda)u_{kk}^2\\
\ge&\delta^2 \theta f^{-1}S_{k-1}(\lambda)\lambda_1^2=\tilde \theta f^{-1}S_{k-1}(\lambda)\lambda_1^2.
\end{align*}
Then by \eqref{FP0}, we have
\begin{align}
 F^{ii}P_{ii}\ge& c_{n,k}h(\log h)'\tilde \theta f^{-1}S_{k-1}(\lambda)\lambda_1^2-C((\log h)')^2PS_{k-1}(\lambda)|Du|^2-CP^{\frac{1}{2}} (\log h)'\notag\\
=&  \frac{1}{2}c_{n,k}h(\log h)'\tilde \theta f^{-1}S_{k-1}(\lambda)\lambda_1^2-C((\log h)')^2PS_{k-1}(\lambda)|Du|^2\notag\\
&+\frac{1}{2}c_{n,k}h(\log h)'\tilde \theta f^{-1}S_{k-1}(\lambda)\lambda_1^2-CP^{\frac{1}{2}} (\log h)'\notag\\
\ge& h^{-1}(\log h)'f^{-1}S_{k-1}(\lambda)(\frac{1}{2}c_{n,k}\tilde\theta(h\lambda_1)^2-CP^2)\notag\\
&+(\log h)'(\frac{1}{2}c_{n,k}\tilde\theta f^{-1}S_{k-1}\lambda_1(h\lambda_1)-CP^{\frac{1}{2}})\notag\\
\ge&(\log h)'(\frac{1}{2}c_{n,k}\tilde\theta f^{-1}S_{k-1}\lambda_1(h\lambda_1)-CP^{\frac{1}{2}})\notag\\
\ge& (\log h)'\frac{1}{4n} c_{n,k}\tilde\theta f^{-1}S_{k-1}\lambda_1(h\lambda_1)\label{FP}\\
=&:c_0(\log h)' f^{-1}S_{k-1}\lambda_1(h\lambda_1),\notag
\end{align}
where we use the following Maclaurin inequality in \eqref{FP}
\begin{align*}
	\frac{S_k(\lambda)}{S_{k-1}(\lambda)}\le \frac{n-k+1}{nk}S_1(\lambda),
	\end{align*}
and we have assumed
 \begin{align}\label{bound1}
 (\lambda_1h)^2>&4 c_{n,k}^{-1} \tilde\theta^{-1}CP^2=:C_1,\\
\lambda_1 h\ge& 8c_{n,k}^{-1} \tilde\theta^{-1}CP^{\frac{1}{2}}:=C_2^{\frac{1}{2}}\label{bound2}.
\end{align}
In this case, inserting \eqref{1}  into \eqref{main}, we have
 \begin{align}
 0\ge F^{ii}G_{ii}\ge&
 \Big(\frac{\varphi''}{\varphi}-3\Big(\frac{\varphi'}{\varphi}  \Big)^2\Big)\sum_{i=1}^nF^{ii}|P_i|^2+\Big(\frac{h''}{h}
 -3\Big(\frac{h'}{h}  \Big)^2\Big)\sum\limits_{i=1}^n F^{ii}|u_i|^2\notag\\
&+\frac{\varphi'}{\varphi}F^{ii}P_{ii}+\frac{kh'}{h}-\lambda_1^{-1}(\tilde f)_{11}
\notag\\
\ge& \Big(\frac{h''}{h}
-3\Big(\frac{h'}{h}  \Big)^2\Big)\sum\limits_{i=1}^n F^{ii}|u_i|^2+\frac{\varphi'}{\varphi}F^{ii}P_{ii}+\frac{kh'}{h}-\lambda_1^{-1}(\tilde f)_{11}\label{main1}.
\end{align}
Combining \eqref{main1} with \eqref{FP} and note that
$\left|\frac{h''}{h}
	-3(\frac{h'}{h}  )^2\right|\le C((\log h)')^2$, we have
\begin{align}
	0\ge&F^{ii}G_{ii} \ge
	\left(\frac{h''}{h}
	-3\left(\frac{h'}{h}  \right)^2\right)\sum\limits_{i=1}^n F^{ii}|u_i|^2+c_0\frac{\varphi'}{\varphi}(\log h)'f^{-1}S_{k-1}\lambda_1(h\lambda_1)+k(\log h)'-\lambda_1^{-1}(\tilde f)_{11}\notag\\
	\ge& -Cf^{-1}S_{k-1}(\lambda) |Du|^2((\log h)')^2+c_0\frac{\tau}{M-P}(\log h)'f^{-1}S_{k-1}\lambda_1(\lambda_1 h)-\lambda_1^{-1}\tilde {f}_{11}\notag\\
	=&h^{-1}(\log h)'f^{-1}S_{k-1}(\lambda)\left(\frac{c_0\tau}{M-P}(\lambda_1 h)^2-CP\right)-\lambda_1^{-1}\tilde {f}_{11}\notag\\
	\ge& \frac{1}{2}h^{-1}(\log h)'f^{-1}S_{k-1}(\lambda)c_0\frac{\tau}{M-P}(\lambda_1 h)^2-\lambda_1^{-1}\tilde {f}_{11},\label{7131}
	\end{align}
where we have assumed
\begin{align}\label{bound3}
	(\lambda_1 h)^2\ge 2c_0^{-1}\tau^{-1}(M-P)CP=:C_3.
	\end{align}

Since $\frac{S_k(\lambda)}{S_{k-1}(\lambda)}\le \frac{n-k+1}{nk}S_1(\lambda)$, by \eqref{7131}, we have
\begin{align}\label{bound4}
	(\lambda_1 h)^2\le Ch((\log h)')^{-1}{\tilde f}_{11}.
	\end{align}

In conclusion, by \eqref{bound1}, \eqref{bound2}, \eqref{bound3} and \eqref{bound4}, we obtain an upper bound of $\lambda_1u^{b}$ as follows
\begin{align}\label{Bound1}
(\lambda_1u^{b})^2\le& C_1+C_2+C_3+Ch((\log h)')^{-1}{\tilde f}_{11}\le \tilde C,
\end{align}
where we use $|D^2 \tilde f|\le C|x|^{-2}$ (see \eqref{f0714}) and $ h((\log h)')^{-1}\le C|x|^2$ (see \eqref{h0714}).

\emph{\textbf{Case2:} $\lambda_k\le \delta \lambda_1$},\\
Since $\lambda_k+\lambda_{k+1}+\cdots\lambda_n=S_1(\lambda_1|12\ldots,k-1)>0$,
we have $-\lambda_n\le (n-k)\lambda_k<n\delta \lambda_1$, thus $|\lambda_i|<n\delta\lambda_1,i=k+1,\ldots,n$.\\
Inserting  \eqref{FP0} into \eqref{main}, we obtain
\begin{align}
	0\ge& F^{ii}G_{ii}\ge 2\lambda_1^{-1}f^{-1}\sum_{i=2}^n S_{k-2}(\lambda|1i)|u_{11i}|^2-
	\sum_{i=2}^n{F^{ii}}\left(\frac{u_{11i}}{u_{11}}\right)^2
	\notag\\
	&+\left(\frac{\varphi''}{\varphi}-\left(\frac{\varphi'}{\varphi}  \right)^2\right)\sum_{i=2}^nF^{ii}|P_i|^2+\left(\frac{h''}{h}-\left(\frac{h'}{h}  \right)^2\right)\sum\limits_{i=2}^n F^{ii}|u_i|^2\notag\\
	&+\left(\frac{\varphi''}{\varphi}-3\left(\frac{\varphi'}{\varphi}  \right)^2\right)F^{11}|P_1|^2+\left(\frac{h''}{h}-3\left(\frac{h'}{h}  \right)^2\right)F^{11}|u_1|^2\notag\\
	&+\frac{\varphi'}{\varphi}c_{n,k}h(\log h)'\sum_{i=1}^nF^{ii}u_{ii}^2-C\frac{\varphi'}{\varphi}((\log h)')^2P\sum_{i=1}^nF^{ii}u_i^2-C\frac{\varphi'}{\varphi}P^{\frac{1}{2}} (\log h)'-\lambda_1^{-1}(\tilde f)_{11}\label{27131}.
\end{align}

\textbf{\emph{Claim}}: \emph{We claim that for sufficiently small $\tau>0$, there exists sufficiently small $\delta>0$ such that  the following holds}
\begin{align}
	&(*):=2\lambda_1^{-1}f^{-1}\sum_{i=2}^n S_{k-2}(\lambda|1i)|u_{11i}|^2-
	\sum_{i=2}^n{F^{ii}}\left(\frac{u_{11i}}{u_{11}}\right)^2
	+\left(\frac{\varphi''}{\varphi}-\left(\frac{\varphi'}{\varphi}  \right)^2\right)\sum_{i=2}^nF^{ii}|P_i|^2\notag\\
	&+(\log h)''\sum\limits_{i=2}^n F^{ii}|u_i|^2-C\frac{\varphi'}{\varphi}((\log h)')^2P\sum_{i=2}^nF^{ii}u_i^2\label{27132}\\
	\ge& 0\notag.
	\end{align}
We will prove the estimate of $\lambda_1 h$ based on the claim above. Indeed, by \eqref{27132} and \eqref{27131}, we have
\begin{align}
	0\ge& F^{ii}P_{ii}\ge (\log\varphi)'c_{n,k}h(\log h)'\sum\limits_{i=1}^n F^{ii}u_{ii}^2\notag\\
	&-C(\log\varphi)'((\log h)')^2F^{11}u_1^2-C(\log\varphi)'(\log h)'-\lambda_1^{-1}\tilde f_{11}\notag\\
	\ge&
	h^{-1}(\log\varphi)'(\log h)'F^{11}\left(\frac{c_{n,k}}{2}(\lambda_1 h)^2-Ch(\log h)'u_1^2\right)\notag\\
	&+(\log\varphi)'\frac{c_{n,k}}{2}h(\log h)'\sum\limits_{i=1}^n F^{ii}u_{ii}^2-C(\log\varphi)'(\log h)'-\lambda_1^{-1}\tilde f_{11}\notag\\
	\ge& h^{-1}(\log\varphi)'(\log h)'F^{11}\left(\frac{c_{n,k}}{2}(\lambda_1 h)^2-CP\right)\notag\\
	&+c_0(\log\varphi)'(\log h)'(\lambda_1h)-C(\log\varphi)'(\log h)'-\lambda_1^{-1}\tilde f_{11}\notag\\
	\ge& \frac{1}{2}c_0(\log\varphi)'(\log h)'(\lambda_1h)-\lambda_1^{-1}\tilde f_{11},\label{27133}
\end{align}
where we use the Maclaurin inequality and assume $\lambda_1 h$ is sufficiently large. From \eqref{27133}, we obtain
\begin{align}
( \lambda_1 h)^2\le Ch((\log h)')^{-1}{\tilde f}_{11}\le \tilde C,\label{Bound2}	
	\end{align}
where we use $|D^2 \tilde f|\le C|x|^{-2}$ (see \eqref{f0714}) and $| h((\log h)')^{-1}\le x|^{2}$ (see \eqref{h0714}).
\\

Now we prove the Claim \eqref{27132}.\\
\emph{Proof of the Claim \eqref{27132}:}
By  Page1037 (3.5) in Chou-Wang\cite{cw2001cpam}, for any sufficiently small $\epsilon_0>0$, there exists $\delta>0$ such that
\begin{align}\label{27134}
	2\lambda_1S_{k-2}(\lambda|1i)-(2-\epsilon_0)S_{k-1}(\lambda|i)>0.
	\end{align}
Now we use \eqref{27134}  to prove the claim by choosing appropriate $\epsilon_0$ and $\delta$. Firstly,  by \eqref{2}, we have
\begin{align}
	\sum\limits_{i=2}^nF^{ii}|(\log h)_i|^2=&-\sum\limits_{i=2}^n|\frac{u_{11i}}{u_{11}}+(\log\varphi)_i|^2\notag\\
	\ge& -2\sum\limits_{i=2}^n|\frac{u_{11i}}{u_{11}}|^2-2\sum\limits_{i=2}^n|(\log\varphi)_i|^2.\label{27135}
\end{align}
\emph{Subcase 1: $k<\frac{n}{2}$ or $\frac{n}{2}<k<n$.}

 We deal with the term $(\log h)'' \sum\limits_{i=2}^nF^{ii}u_i^2$ as follows
\begin{align}\label{27136}
	(\log h)''\sum\limits_{i=2}^n F^{ii}u_i^2=(\log h)''((\log h)')^{-2}\sum\limits_{i=2}^n|\frac{u_{11i}}{u_{11}}+(\log\varphi)_i|^2
	=-\frac{|n-2k|}{n}\sum\limits_{i=2}^n|\frac{u_{11i}}{u_{11}}+(\log\varphi)_i|^2\notag	\\
\ge -(1+\tau_1)\frac{|n-2k|}{n}\sum\limits_{i=2}^n|\frac{u_{11i}}{u_{11}}|^2-(1+\tau_1^{-1})\frac{|n-2k|}{n}\sum\limits_{i=2}^n|(\log\varphi)_i|^2,
\end{align}
where $\tau_1>0$ is a sufficiently small constant.\\
Inserting \eqref{27135} and \eqref{27136} into \eqref{27132}, we have
\begin{align}
(*)\ge& \left(1-\epsilon_0-(1+\tau_1)\frac{|n-2k|}{n}-2\tau CP(M-P)^{-1}\right)|\frac{u_{11i}}{u_{11}}|^2\notag\\
&+
\left(\frac{\varphi''}{\varphi}-\left(1+(1+\tau_1^{-1})\frac{|n-2k|}{n}+2\tau CP(M-P)^{-1}\right)((\log \varphi)')^2\right)F^{ii}|P_i|^2\notag\\
\ge& \left(1-\epsilon_0-(1+\tau_1)\frac{|n-2k|}{n}-2C\tau \right)|\frac{u_{11i}}{u_{11}}|^2\notag\\
&+
\left(\frac{\varphi''}{\varphi}-\left(1+(1+\tau_1^{-1})\frac{|n-2k|}{n}+2C\tau \right)((\log \varphi)')^2\right)F^{ii}|P_i|^2\label{27137}.
\end{align}
Since   $|n-2k|<n$,  we  can choose $\tau_1$ such that $b=(1+\tau_1)\frac{|n-2k|}{n}<1$ and then we choose $\epsilon_0=\frac{1-b}{2}$ and $M>\frac{2C_1}{1-b}$ such that the first term of \eqref{27137} is nonnegative. At last, since $\frac{\varphi''}{\varphi}=(1+\tau^{-1})((\log\varphi)')^2$,  the second term of \eqref{27137} is nonnegative if we choose positive small constant $\tau$ as follows
\begin{align}\label{0727tau1}
	\tau<\min\Big\{\frac{1-b}{4C}, \frac{n}{(1+\tau_1^{-1})|n-2k|+2nC}\Big\}.
	\end{align}

\emph{Subcase 2: $k=\frac{n}{2}$.}\\
In this case $(\log h)''=0$,
inserting \eqref{27135} into \eqref{27132} and using \eqref{27134} with $\epsilon_0=\frac{1}{2}$, we have
\begin{align*}
	(*)\ge& \left(\frac{1}{2}-{2\tau CP}(M-P)^{-1}\right)|\frac{u_{11i}}{u_{11}}|^2+\left(\frac{\varphi''}{\varphi}-(1+2\tau CP(M-P)^{-1})((\log \varphi)')^2\right)F^{ii}|P_i|^2\\
	\ge& 0,
	\end{align*}
where we choose the positive small constant $\tau$ satisfying
\begin{align}\label{0727tau2}
	\tau<\frac{1}{4C+4}.
	\end{align}
Then we finish the proof of the Claim.
Combining \eqref{Bound1} with \eqref{Bound2}, we get the estimate of $\lambda_1 h$.

\end{proof}
\subsubsection{\textbf{Second order estimate on the boundary $\partial \Omega_R$}}.

\textbf{Step1: tangential derivative estimates.}\\

We first prove the tangential derivative estimate on $\partial\Omega$.
For any  $x_0\in\partial\Omega$, we choose the coordinate such that $x_0=0$, $\partial\Omega\bigcap B_\delta (0)=(x', \rho(x'))$, $\rho(0)=0$ and $\nabla \rho(0)=0$.
Since $u(x', \rho(x'))$ is constant,
we have
\begin{align*}
	|u_{\alpha\beta}(0)|=&|u_{n}\rho_{\alpha\beta}(0)|
	\le C|Du|(0)\le C.
	\end{align*}
Next we can prove $S_{k-1}(u_{\alpha\beta}(0))\ge c_1>0$ as that in Section 3 (see \eqref{0717c2})

For any $x_0\in \p B_R$, we choose the coordinate such that $x_0=(0,\cdots,
 0, -R)$, then near $x_0$, $\p B_R$ is locally represented by $x_n=-(R^2-|x'|^2)^{\frac{1}{2}}$.\\
 Since $u|_{\partial B_R}=constant$, we have
 \begin{align}\label{4627141}
 u_{\alpha\beta}(x_0)=&-u_n(x_0)\frac{\p ^2 x_n}{\p x_{\alpha}\p x_{\beta}}(x_0)=-R^{-1}u_n(x_0)\delta_{\alpha\beta}\notag\\
 =&R^{-1} u_{\nu}(x_0)\delta_{\alpha\beta}.
 \end{align}

 Since we have the boundary gradient estimate on $\partial B_R$ (see \eqref{457141}, \eqref{0714gradient2.2} and \eqref{0714gradient3.2}),
 \begin{align*}
CR^{-\frac{n-k}{k}}\ge u_{\nu}(x)\ge cR^{-\frac{n-k}{k}},
\end{align*}
then by \eqref{4627141}, we have
\begin{align}
|u_{\alpha\beta}(x_0)|\le& CR^{-\frac{n}{k}}\\
\{u_{\alpha\beta}(x_0)\}\ge& c R^{-\frac{n}{k}} \{\delta_{\alpha\beta}\}.
\end{align}

\emph{\textbf{Step2: { tangential-normal derivative estimates $\partial \Omega_R$}}}

 For any $x_0\in \partial B_R$,
choose the coordinate such that $x_0=(0,\cdots,0,-R)$, $\partial B_R\cap B_{\frac{1}{2} R}(x_0)$ is represented by
\begin{align*}
	x_n=\rho(x')=-(R^2-|x'|^2)^{\frac{1}{2}},
\end{align*}
Consider the tangential operator $T_{\alpha}=\left(x_\alpha \partial_{n}-x_n \partial_{\alpha}\right)$,$1\le \alpha\le n-1$. Since $u(x',\rho(x'))$ is constant, we have
\begin{align*}
	0=&	u_{\alpha}+u_{n}\rho_{\alpha}
	=u_{\alpha}-x_{\alpha}\rho^{-1}u_n
\end{align*}
Then {on} $\partial B_R\cap B_{\frac R 2}(x_0)$, we have
\begin{align*}
	T_{\alpha} u=x_\alpha u_{n}-\rho u_{\alpha}=0.
\end{align*}
We consider the function
\begin{align*}
	w=A_0(1+\frac{x_n}{R})\pm R^{\frac{n-2k}{k}}T_{\alpha } u \ \text{in}\ B_R\cap B_{\frac{R}{2}}(x_0),
\end{align*}

Obviously, $w(x_0)=0$. Since $T_\alpha u=0$ on ${\partial B_{R}\setminus B_{\frac R2}(x_0)}$, we have $w|_{\partial B_{R}\setminus B_{\frac R2}(x_0)}\ge 0$. \\
Since $R^{\frac{n-2k}{k}}|T_{\alpha} u|\le C_1R^{\frac{n-2k}{k}}|x||Du|\le C$,  choosing $A_0>8C$, we have \begin{align}
	w\ge \frac18 A_0-C>0 \quad \text{on}\quad {B_R\cap \partial B_{\frac{1}{2}R}(x_0)}.
\end{align}
Next we show $F^{ij}w_{ij}=0$. Indeed, firstly, recall $ D \tilde f=D\log f=-(\frac{n}{2}+1)D\log(|x|^2+\varepsilon)$, then we have
\begin{align}
	F^{ij}(T_{\alpha} u)_{ij}=x_{\alpha}\tilde f_{n}-x_{n}\tilde f_{\alpha}=0.
\end{align}
Thus we have
\begin{align*}
	F^{ij}w_{ij}=0.
\end{align*}
By maximum principle, $\min\limits_{\overline{B_R\cap B_{\frac R2}(x_0)}}w=w(x_0)=0$. Then we have
\begin{align*}
	0\le w_n(x_0)=R^{-1}A_0\pm R^{\frac{n-k}{k}} u_{\alpha n}(x_0).
\end{align*}
Then
$|u_{\alpha n}(x_0)|\le A_0  R^{-\frac nk}$ and thus we have the uniform tangential-normal derivative estimates on $\partial B_R$.
	
For any $x_0$, since $S_k(D^2\underline{u}^{i,\varepsilon})\ge \epsilon_0>0$ near $\partial \Omega$, we can prove the tangential-normal derivative estimates on $\partial \Omega$ similar as that in Section 3.

\emph{\textbf{Step3: double normal derivative estimates $\partial \Omega_R$}}\\
We can choose the coordinate at $x_0$ such that $u_{n}(x_0)=|Du|$ and $\{u_{\alpha\beta}(x_0)\}_{1\le \alpha,\beta\le n-1}$ is diagonal.

 If $x_0\in \partial B_R$, we have
\begin{align*}
	u_{nn}c_0R^{-\frac{n(k-1)}{k}}\le u_{nn}(x_0)S_{k-1}(u_{\alpha\beta}(x_0))=&S_k(D^2 u(x_0))-S_{k}(u_{\alpha\beta}(x_0))
	+\sum_{i=1}^{n-1}u_{in}^2S_{k-2}(u_{\alpha\beta})\\
	\le& f+C_{n-2}^{k-2}M_{21}^{k-2}M_{22}^2R^{-n}\\
	\le& C R^{-n}.
\end{align*}
This gives $u_{nn}\le C R^{-\frac{n}{k}}$. On the other hand,  $u_{nn}\ge -\sum\limits_{i=1}^{n-1}u_{ii}\ge -(n-1)M_{21}R^{-\frac{n}{k}}$.
Then we have $|u_{nn}(x_0)|\le CR^{-\frac{n}{k}}$.


If $x_0\in \partial \Omega$, since $S_{k-1}(u_{\alpha\beta}(x_0))\ge c_1$  which can be proved similar as that in Section 3, then we have $|u_{nn}(x_0)|\le C$.

In conclusion, we obtain $|D^2 u(x)|\le C|x|^{-\frac{n}{k}}$ on the boundary $\partial\Omega_R $ and thus $|D^2u|(x)\le C|x|^{-\frac nk}$ for any $x\in \overline\Omega_R$.
\section{Proof of Theorem \ref{main07201}, Theorem \ref{main07202} and Theorem \ref{main07203}}
\subsection{Uniqueness}
The uniqueness follows from the comparison principle for $k$-convex solutions of the $k$-Hessian equation in bounded domains in Lemma \ref{comparison0718} by Wang-Trudinger\cite{trudingerwang1997tmna} (see also \cite{trudinger1997cpde, urbas1990indiana}).\\

\textbf{Case1: $k<\frac{n}{2}$}\\
Let $u_1,u_2$ be solutions of the $k$-Hessian equation. For any $x_0\in \Omega^c$, we want to prove $u_1(x_0)\ge u_2(x_0)$. Indeed, since $\lim_{|x|\rightarrow \infty}u_i=0$, for any $\epsilon>0$, there esists sufficiently large  $R$  such that $x_0\in B_R(0)$ and  $u_1\ge u_2-\epsilon$ on $\partial B_R(0)$. Note we also have $u_1=u_2=0$ on $\partial \Omega$, by comparison theorem in $\Omega_R$, we then have $u_1\ge u_2-\epsilon$ in $\Omega_R$. Let $\epsilon$ go to $0$, we have $u_1(x_0)\ge u_2(x_0)$ and thus $u_1\ge u_2$ in $\Omega^c$. Similarly, we can prove  $u_2\ge u_1$ in $\Omega^c$. Then we have $u_1=u_2$ in $\Omega^c$ and thus prove the uniqueness part.\\

\textbf{Case 2: $k>\frac{n}{2}$}\\
For any $x_0\in \Omega^c$, we want to prove $u_1(x_0)\ge  u_2(x_0)$. Indeed, for any $t\in(0, 1)$, since $u_1-tu_2=(1-t)|x|^{\frac{2k-n}{k}}+O(1)$ when $|x|\rightarrow \infty$, there exists sufficiently large $R$ such that

\begin{align}
	x_0\in B_R(0) \quad \text{and}\quad   u_1>tu_2 \quad\text{on}\quad \partial B_R(0)
\end{align}

Note that we also have $u_1=1>tu_2$ on $\partial\Omega$. By comparison theorem, we then have  $u_1\ge tu_2$ in $\Omega_R$. In particular $u_1(x_0)\ge tu_2(x_0)$. Let $t$ tend to $1$, we have  $u_1(x_0)\ge u_2(x_0)$ and thus  $u_1\ge u_2$ in $\Omega^c$.

 Similarly, we have $u_2\ge u_1$. Then we have $u_1=u_2$ in $\Omega^c$ and thus prove the uniqueness part.\\

\textbf{Case 3: $k=\frac{n}{2}$}\\
Let $x_0\in \Omega^c$. For any $t\in(0, 1)$, since $u_1-tu_2=(1-t)\log|x|+O(1)$ when $|x|\rightarrow \infty$, there exists sufficiently large $R$ such that

\begin{align}
	x_0\in B_R(0) \quad \text{and}\quad   u_1> tu_2 \quad\text{on}\quad \partial B_R(0)
\end{align}

Since  $u_1= tu_2=0$ on $\partial\Omega$, by comparison theorem, we then have  $u_1\ge tu_2$ in $B_R(0)$. In particular $u_1(x_0)\ge tu_2(x_0)$. Let $t$ tend to $1$, we have  $u_1(x_0)\ge u_2(x_0)$ and thus  $u_1\ge u_2$ in $\Omega^c$.

 Similarly, we have $u_2\ge u_1$. Then we have $u_1=u_2$ in $\Omega^c$ and thus prove the uniqueness part.

\subsection{Existence and $C^{1,1}$-estimates}
The existence follows from the uniform $C^2$-estimates for $u^{\varepsilon, R}$.

\textbf{Case 1: $k<\frac{n}{2}$}\\
For any fixed sufficiently small $\varepsilon>0$,
 by Guan \cite{guan1994cpde}, $|u^{\varepsilon, R}|_{C^m(\Omega_{K_0})}\le C(\epsilon, K_0,m)$ for any $K_0>R_0$  and  $m\ge 0$ (we always assume $\Omega\subset\subset B_{\frac{R_0}{2}}$). Then there exists a subsequence $u^{\varepsilon, R_i}$ converging smoothly to a strictly $k$-convex $u^\varepsilon$ in $K$ and $u^{\infty}\in C^{\infty}(\Omega^c)$ satisfies
 \begin{align}
 	\left\{\begin{aligned}
 		S_k(D^2 u^\varepsilon)=f^{1,\varepsilon}\qquad\text{in}\ \ \Omega^c,\\
 		u^\varepsilon=-1, \qquad\text{on}\ \ \p\Omega.
 		\end{aligned}
 	\right.
 	\end{align}

Moreover, by Theorem \ref{apu10720}, we get
 \begin{align*}
 	\left\{
 	\begin{aligned}
 		C^{-1}|x|^{-\frac{n-2k}{k}}\le& -u^{\varepsilon}(x)\le C|x|^{-\frac{n-2k}{k}},\\
 		C^{-1}|x|^{-\frac{n-k}{k}}\le|Du^{\varepsilon}|(x)\le& C|x|^{-\frac{n-k}{k}},\\
 		|D^2u^{\varepsilon}|(x)\le& C|x|^{-\frac{n}{k}},
 	\end{aligned}
 	\right.
 \end{align*}
 Thus there exits a subsequence $u^{\epsilon_i}$ converges to $u$ in $C^{1,\alpha}_{loc}$ such that $u\in C^{1,1}(\Omega^c)$ is the $k$-convex solution of the k-Hessian equation \eqref{case1Equa1.1} and satisfies the estimates \eqref{decay10720}.

 \textbf{Case 2: $k>\frac{n}{2}$}\\
 For any fixed sufficiently small $\varepsilon>0$,
 by Guan \cite{guan1994cpde}, $|u^{\varepsilon, R}|_{C^m(\Omega_{K_0})}\le C(\epsilon, K_0,m)$ for any $K_0>R_0$  and  $m\ge 0$ . Then there exists a subsequence $u^{\varepsilon, R_i}$ converging smoothly to a strictly $k$-convex $u^\varepsilon$ in $K$ and $u^{\varepsilon}\in C^{\infty}(\Omega^c)$ satisfies
 \begin{align}
 	\left\{\begin{aligned}
 		S_k(D^2 u^\varepsilon)=f^{2,\varepsilon}\qquad\text{in}\ \ \Omega^c,\\
 		u^\varepsilon=1, \qquad\text{on}\ \ \p\Omega.
 	\end{aligned}
 	\right.
 \end{align}

 Moreover, by Theorem \ref{apu20720}, we get
 \begin{align*}
 	\left\{
 	\begin{aligned}
 		 |u^{\varepsilon}(x)-|x|^{\frac{2k-n}{k}}|\le& C,\\
 	C^{-1}|x|^{-\frac{n-k}{k}}	|Du^{\varepsilon}|(x)\le& C|x|^{-\frac{n-k}{k}},\\
 		|D^2u^{\varepsilon}|(x)\le& C|x|^{-\frac{n}{k}},
 	\end{aligned}
 	\right.
 \end{align*}
 Thus there exits a subsequence $u^{\epsilon_i}$ converges to $u$ in $C^{1,\alpha}_{loc}$ such that $u\in C^{1,1}(\Omega^c)$ is the $k$-convex solution of the k-Hessian equation \eqref{case2Equa1.2} and satisfies the estimates \eqref{decay20720}.

 \textbf{Case 3: $k=\frac{n}{2}$}\\
 For any fixed sufficiently small $\varepsilon>0$,
 by Guan \cite{guan1994cpde}, $|u^{\varepsilon, R}|_{C^m(\Omega_{K_0})}\le C(\epsilon, K_0,m)$ for any $K_0>R_0$  and  $m\ge 0$ . Then there exists a subsequence $u^{\varepsilon, R_i}$ converging smoothly to a strictly $k$-convex $u^\varepsilon$ in $K$ and $u^{\varepsilon}\in C^{\infty}(\Omega^c)$ satisfies
 \begin{align}
 	\left\{\begin{aligned}
 		S_k(D^2 u^\varepsilon)=f^{3,\varepsilon}\qquad\text{in}\ \ \Omega^c,\\
 		u^\varepsilon=0, \qquad\text{on}\ \ \p\Omega.
 	\end{aligned}
 	\right.
 \end{align}

 Moreover, by Theorem \ref{apu30720}, we get
 \begin{align*}
 	\left\{
 	\begin{aligned}
 		|u^{\varepsilon}(x)-\log|x||\le& C,\\
 		C^{-1}|x|^{-1}\le|Du^{\varepsilon}|(x)\le& C|x|^{-1},\\
 		|D^2u^{\varepsilon}|(x)\le& C|x|^{-2},
 	\end{aligned}
 	\right.
 \end{align*}
 Thus there exits a subsequence $u^{\epsilon_i}$ converges to $u$ in $C^{1,\alpha}_{loc}$ such that $u\in C^{1,1}(\Omega^c)$ is the $k$-convex solution of the k-Hessian equation \eqref{case3Equa1.3} and satisfies the estimates \eqref{decay30720}.

\section{Almost monotonicity formula along the level set of the approximating solution}
 Agostiniani-Mazzieri \cite{AM2020CVPDE} proved an monotonicity formula along the level set of the solution of the following problem
 \begin{align}
 	\left\{
 	\begin{aligned}
 		\Delta u=0 \ \text{in} \ \Omega^c\\
 		u=-1  \ \text{on} \ \p\Omega\\
 		\lim\limits_{|x|\rightarrow\infty}u(x) =0.
 		\end{aligned}
 	\right.
 	\end{align}
In our setting, note that $u$ is only $C^{1,1}$, we consider similar quantity on the level set of $u^{\varepsilon}$ since $u^{\varepsilon}$ is smooth and $|Du^{\varepsilon}|\equiv |x|^{1-\frac{n}{k}}$.

Firstly, as an application of the $C^0$ estimates of $u^{\epsilon}$, we prove the following property.
 \begin{lemma}
 	Assume $k<\frac{n}{2}$.
 	\begin{align}\label{07230}
 	\lim_{\epsilon\rightarrow 0}\int_{R^n\backslash  \bar \Omega}{S_k^{ij}(D^2u^{\epsilon})u^{\epsilon}_iu^{\epsilon}_j}=
 		&\int_{\partial\Omega}{|Du|^{k}H_{k-1}(\kappa)dA},
 	\end{align}
 \end{lemma}
\begin{remark}
	We may call $\int_{\partial\Omega}{|Du|^{k}S_{k-1}(\kappa)dA}$ as the $k$-Capacity of $\Omega$, since
	when $k=1$, $Cap(\Omega)=\int_{\partial\Omega}{|Du|}dA$. The left hand side may be $\infty$ when $k\ge \frac{n}{2}$.
	\end{remark}
 \begin{proof}

 Let $\Omega_t:= \{x\in \Omega^c:u^{\varepsilon}<t\}$ and $S_t=\{x\in\Omega^c: u^{\varepsilon}(x)=t\}$ with $t\in [-1, 0)$.
%

 Since $|Du^{\varepsilon}|>0$, the level set $S_t:= \{u^{\varepsilon}=t\}$ with $t\in [-1, 0)$ is a smooth closed hypersurface. Let $\nu$ be the outward unit normal vector of $S_t$ and then  we  have $\nu=\frac{Du}{|Du|}$.
 	By the divergence free property of the k-Hessian operator: $D_i S^{ij}=0$, we have,
 	\begin{align}\label{omegat0721}
 		&\int_{\partial \Omega_t}{S_k^{ij}(D^2u^{\varepsilon})u^{\varepsilon}_j\nu_{i}}-\int_{\partial \Omega}{S_k^{ij}(D^2u^{\varepsilon})u^{\varepsilon}_j\nu_{i}}\notag\\
 		&=\int_{\Omega_t\backslash\bar\Omega}{D_i \left(S_k^{ij}(D^2u^{\varepsilon})u^{\varepsilon}_j\right)}dx\notag=
 		\int_{\Omega_t\backslash\bar\Omega}{
 			S_k^{ij}(D^2u^{\varepsilon})u^{\varepsilon}_{ij}} dx\\
 		&=\int_{S_t\backslash\bar\Omega}k
 		{S_k(D^2u^{\varepsilon})}dx=kc_{n,k}\varepsilon^2\int_{\Omega_t\backslash\bar\Omega}{\left(|x|^2+\varepsilon^2\right)^{-\frac{n}{2}-1}}.
 	\end{align}

 Since $C^{-1}|x|^{-\frac{n-2k}{k}}\le |u^{\varepsilon}|\le C |x|^{-\frac{n-2k}{k}}$ in $ \Omega^c$, for any $t\in[-1, 0)$,
 \begin{align}
 	B_{C^{-1}t^{-\frac{k}{n-2k}}}\subset \Omega_t\subset B_{Ct^{-\frac{k}{n-2k}}},
 \end{align}
 then  we have
 \begin{align}
 	\left|\int_{ S_t}{S_k^{ij}(D^2u)u^{\varepsilon}_j\nu_{i}}-\int_{\partial \Omega}{S_k^{ij}(D^2u^{\varepsilon})u^{\varepsilon}_j\nu_{i}}\right|\le  C\varepsilon^2(r_0^{-2}-|t|^{\frac{2k}{n-2k}}).
 	\end{align}
 	Thus  for any $t\in [-1, 0)$, we have
 	\begin{align}\label{equality1}
 		\int_{S_t}{S_k^{ij}(D^2u^{\varepsilon})u^{\varepsilon}_j\nu_{i}}=\int_{\partial \Omega}{S_k^{ij}(D^2u^{\varepsilon})u^{\varepsilon}_j\nu_{i}}	
 		+O(\varepsilon^2)(r_0^{-2}-|t|^{\frac{2k}{n-2k}}).
 	\end{align}
 	Then by the coarea formula, we have
 	\begin{align}
 		&\int_{R^n\backslash  \bar \Omega}{S_k^{ij}(D^2u^{\varepsilon})u^{\varepsilon}_iu^{\varepsilon}_j} \notag\\
 		=&\int_{-1}^0\int_{S_t}{S_k^{ij}(D^2u^{\varepsilon})u^{\varepsilon}_j\frac{u^{\varepsilon}_i}{|Du^{\epsilon}|}dA(t)} dt\ \ \text{\emph{\text{(Coarea Formula)}}}\notag \\
 		=&\int_{\{u=-1\}=\partial\Omega}{S_k^{ij}(D^2u^{\varepsilon})u^{\varepsilon}_j\frac{u^{\varepsilon}_i}{|Du^{\varepsilon}|}dA} +O(\varepsilon^2)(\emph{\textbf{\text{By} \eqref{equality1}}})\notag\\
 		=&\int_{\partial\Omega}{|Du^{\varepsilon}|^{k}H_{k-1}(\kappa)dA}+O(\varepsilon^2),
 	\end{align}
 where we use $H_{k-1}(\kappa)=|Du^{\varepsilon}|^{-k-1}S_k^{ij}(D^2 u^{\varepsilon})u^{\varepsilon}_iu^{\varepsilon}_j$.

 	
 	Let $\varepsilon$ tend to $0$ and note that $|Du^{\varepsilon}|$ tends to $|Du|$, we have
 	\begin{align*}
 		\lim_{\varepsilon\rightarrow 0}\int_{R^n\backslash  \bar \Omega}{S_k^{ij}(D^2u^{\varepsilon})u^{\varepsilon}_iu^{\epsilon}_j}=\int_{\partial\Omega}{|Du|^{k}H_{k-1}(\kappa)dA}.
 	\end{align*}
  \end{proof}

 By the uniform $C^2$ estimates and  positive lower bound  of $u^{\epsilon}$, we can estimate $|S_t|$, where $S_{t}=\{x\in \mathbb R^n\setminus \Omega: u(x)=t\}$.

 \begin{lemma}\label{st0723}
There exits uniform constant $C$ such that
 	\begin{align}
 		|S_t|	\le&
 		\left\{\begin{aligned} C|t|^{-\frac{k(n-1)}{n-2k}}\  \ \ \text{for any }\ 	\ t\in (0,-1]\ \text{if} \ k<\frac{n}{2},\\
 			C|t|^{\frac{k(n-1)}{2k-n}} \ \ \text{for any }\ 	 t\in [1,\infty)\ \text{if} \ k>\frac{n}{2},\\
 			 Ce^{(n-1)t}\ \text{for any }\ t\in [0,\infty)\ \text{if} \ k=\frac{n}{2}.
 		\end{aligned}
 		\right.
 		\end{align}
 	\end{lemma}
 	 \begin{proof}
 	 	Since
	\begin{align*}
	 |S_t|-|\partial \Omega|=\int_{\Omega_t}\mathrm{div}\Big(\frac{u^{\varepsilon}_i}{|Du^{\varepsilon}|}\Big)dx
 \end{align*}
By the uniform $C^2$-estimates and the uniform lower bound of $u^{\varepsilon}$, we finish the proof.
 \end{proof}

We define the following quantity

\begin{align}
	I_{a, b, k}(t):=
	\int_{S_t}
	g^a(u^{\varepsilon})|Du^{\varepsilon}|^{b-k}S_k^{ij}(D^2u^{\varepsilon})u^{\varepsilon}_iu^{\varepsilon}_j,
\end{align}
where $g(u)$ s defined by
\vspace*{-5pt}
\begin{align}
	g(u)=\left\{
	\begin{aligned}
		&(-u)^{\frac{n-k}{2k-n}}, k<\frac{n}{2},\\
		&u^{\frac{n-k}{2k-n}}, \ \quad k>\frac{n}{2},\\
	&{e}^{u},\ \qquad k=\frac{n}{2}.
		\end{aligned}
	\right.
	\end{align}
We choose $a=b-k+1$ and one can see that $I_{a,b,k}(t)$ is uniformly bounded from the $C^2$ estimates of $u^{\varepsilon}$ and the lower bound of $|Du^{\varepsilon}|$.\\
When $k=1$ and $a=b$, $I_{a,b,k}(t)$ is exactly the one in \cite{AM2020CVPDE}. \\
We define
\begin{align}	J_{a+a_0,b,k}(t,t_0):=-g^{a_0}(t) I'_{a,b,k}(t)+g^{a_0}(t_0)  I'_{a,b,k}(t_0)
	\end{align}.
We prove the following useful inequalities along the level set of $u^{\varepsilon}$.
\begin{lemma}\label{0725mono1} Let $u^{\varepsilon}$ be the solution of the approximating k-Hessian equation with $a=b-k+1$.
	We have the following inequalities

		\begin{align}
		J_{a+a_0,b,k}(t,t_0)
		\ge& -ba\int_{t}^{t_0}\int_{S_{s}}
	g^{a+a_0}|Du^{\varepsilon}|^{b-k-1}\frac{H_k}{H_{k-1}}S_kdAds-(b+1)\int_{S_t}\Big(g^{a+a_0}|Du^{\varepsilon}|^{b-k}S_k\Big)dA\notag\\
		&+a\int_{t}^{t_0}\int_{S_{s}}g^{a+a_0}|Du^{\varepsilon}|^{b-1}H_{k-1}^{-1}\Big(c_{n,k}H_{k}^2-(k+1)H_{k-1}H_{k+1}\Big)dAds\notag\\
		&+a\int_{t}^{t_0}\int_{S_{s}}
		g^{a+a_0}
		|Du^{\varepsilon}|^{b-1}\mathcal{L}.
		\end{align}
	where $a_0, b, c_{n,k}=\frac{k(n-k-1)}{n-k}$ and the functions $\mathcal{L}$ are choosing as follows
\begin{enumerate}[{(i)} ]
		\item	If $1\le k<\frac{n}{2}$, we require  $-1\le t<t_0<0$, $a_0=-2\frac{n-2k}{n-k}, b\ge c_{n,k}$ and
		$\mathcal{L}=(b-c_{n,k})\Big(
		\frac{n-k}{n-2k}|D\log u^{\varepsilon}|-\frac{H_k}{H_{k-1}}\Big)^2$
	\end{enumerate}
		
\begin{enumerate}[{(ii)} ]
	\item	If $k=\frac{n}{2}$, we require  $0\le t<t_0<\infty$, $a_0=0$ $b\ge \frac{n}{2}-1$ and  $\mathcal{L}=a\Big(
	|D u^{\varepsilon}|-\frac{H_k}{H_{k-1}}\Big)^2$.
\end{enumerate}
\begin{enumerate}[{(iii)} ]
	\item	If $n>k>\frac{n}{2}$, we require  $1\le t<t_0<\infty$ and $a_0=2\frac{2k-n}{n-k}$, $b\ge k-1$ and $\mathcal{L}=(b-c_{n,k})\Big(
	\frac{n-k}{n-2k}|D\log u^{\varepsilon}|-\frac{H_k}{H_{k-1}}\Big)^2$.
\end{enumerate}

	\end{lemma}
\begin{proof}
		\textbf{For simplicity, we use $u$ instead of $u^{\varepsilon}$ and $S_k$ intead of $S_k(D^2u^{\varepsilon})$ during the proof.}
		
	By  the divergence theorem and the divergence free property of the $k$-Hessian operator i.e. $\sum\limits_{j=1}^nD_jS_k^{ij}=0$, we have
	\begin{align}
		I_{a,b,k}(t_0)-I_{a,b,k}(t)=&\int\limits_{\Omega_{t_0}\setminus\Omega_t}D_j\left(g^a|Du|^{b+1-k}S_{k}^{ij}u_i\right)\notag\\
		=&a\int_{\Omega_{t_0}\setminus\overline\Omega}g^{a-1}g'|Du|^{b+1-k}S_k^{ij}u_iu_j\notag\\
		&+(b+1-k)\int_{\Omega_{t_0}\setminus\overline\Omega}g^{a}|Du|^{b-k-1}S_k^{ij}u_iu_lu_{ij}+k\int_{\Omega_{t_0}\setminus\overline\Omega}g^a|Du|^{b+1-k}S_k\notag\\
		=&a\int_{\Omega_{t_0}\setminus\overline\Omega}g^{a-1}g'|Du|^{b+1-k}S_k^{ij}u_iu_j\notag\\
		&-(b+1-k)\int_{\Omega_{t_0}\setminus\overline\Omega}g^{a}|Du|^{b-k-1}S_{k+1}^{ij}u_iu_{j}+(b+1)\int_{\Omega_{t_0}\setminus\overline\Omega}g^a|Du|^{b+1-k}S_k\notag\\
		=&a\int_{t}^{t_0}\int_{S_{s}}g^{a-1}g'|Du|^{b-k}S_k^{ij}u_iu_j-(b+1-k)\int_{t}^{t_0}\int_{S_{s}}g^a|Du|^{b-k-2}S_{k+1}^{ij}u_iu_{j}\notag\\
				&+(b+1)\int_{t}^{t_0}\int_{S_s}g^{a}|Du|^{b-k}S_k,\label{0726It}
	\end{align}
where we use $S_k^{ij}u_iu_lu_{lj}=|Du|^2S_k-S_{k+1}^{ij}u_iu_{j}$ and the coarea formula.

Then
	
	The derivative of $I_{a,b, k}(t)$ is
	\begin{equation}\label{0726I}
		\begin{aligned}
			I'_{a,b,k}(t)=&a\int_{S_t}g^{a-1}g'|Du|^{b-k}S_k^{ij}u_iu_j\\
			&-(b+1-k)\int_{S_t}g^{a}|Du|^{b-k-2}S_{k+1}^{ij}u_iu_{j}+E_{a,b,k}(t),
		\end{aligned}
	\end{equation}
where $E_{a,b,k}(t)=(b+1)\int_{S_t}g^{a}|Du|^{b-k}S_k$.

Then we have
\begin{align}\label{10724}
	&J_{a,b,k}(t,t_0):=-g^{a_0}(t) I'_{a,b,k}(t)+g^{a_0}(t_0) I'_{a,b,k}(t_0)\notag\\
	=&a\int_{\Omega_{t_0}\setminus\Omega_t}D_j\Big(g^{a+a_0-1}g'|Du|^{b-k+1}S_k^{ij}u_i\Big)\notag\\
	&-(b-k+1)(I_{a+a_0,b-1,k+1}(t_0)-I_{a+a_0,b-1,k+1}(t))+E_{a+a_0,b,k}(t_0)-E_{a+a_0,b,k}(t)
	\end{align}
Firstly we have
\begin{align}\label{30724}	
&\int_{\Omega_{t_0}\setminus\overline\Omega_t}D_j\Big(g^{a+a_0-1}g'|Du|^{b-k+1}S_k^{ij}u_i\Big)dx\notag\\
=&	\int_{t}^{t_0}\int_{S_s}\Big((g^{a+a_0-1}g'\Big)'|Du|^{b-k}S_k^{ij}u_iu_j dAds
\notag\\
&+(b-k+1)\int_{t}^{t_0}\int_{S_s}\Big(g^{a+a_0-1}g'|Du|^{b-k-2}S_k^{ij}u_iu_lu_{lj}\Big)dAds+k\int_{t}^{t_0}\int_{S_s}\Big(g^{a+a_0-1}g'|Du|^{b-k}S_k\Big)dAds\notag\\
=&\int_{t}^{t_0}\int_{S_s}\Big((g^{a+a_0-1}g')'|Du|^{b-k}S_k^{ij}u_iu_j\Big)dAds\notag\\
&-(b-k+1)\int_{t}^{t_0}\int_{S_s}\Big(g^{a+a_0-1}g'|Du|^{b-k-2}S_{k+1}^{ij}u_iu_j\Big)dAds+(b+1)\int_{t}^{t_0}\int_{S_s}\Big(g^{a+a_0-1}g'|Du|^{b-k}S_k\Big)dAds\notag\\
=&\int_{t}^{t_0}\int_{S_s}\Big((g^{a+a_0-1}g')'|Du|^{b+1}H_{k-1}\Big)dAds
-(b-k+1)\int_{t}^{t_0}\int_{S_s}\Big(g^{a+a_0-1}g'|Du|^{b}H_{k}\Big)dAds\notag\\
&+(b+1)\int_{t}^{t_0}\int_{S_s}\Big(g^{a+a_0-1}g'|Du|^{b-k}S_k\Big)dAds,
	\end{align}
where we use the identity $H_{m-1}|Du|^{m+1}={S_{m}^{ij}}{u^{i}u^{j}}$ for $m\in \{1,2,\cdots, n\}$(see Lemma \ref{07261}).

For the term $I_{a+a_0,b-1,k+1}(t_0)-I_{a+a_0,b-1,k+1}(t)$, similar as the calculation of \eqref{0726It}, we have
\begin{align}\label{20724}
	I_{a+a_0,b-1,k+1}(t_0)-&I_{a+a_0,b-1,k+1}(t)\notag\\
	=&(a+a_0)\int_{t}^{t_0}\int_{S_{s}}g^{a+a_0-1}g'|Du|^{b-k-2}S_{k+1}^{ij}u_iu_j dAds\notag\\
	&-(b-1-k)\int_{t}^{t_0}\int_{S_{s}}g^{a+a_0}|Du|^{b-k-4}S_{k+2}^{ij}u_iu_{j}\notag\\
	&+b\int_{t}^{t_0}\int_{S_s}g^{a+a_0}|Du|^{b-k-2}S_{k+1}.
	\end{align}
Next we deal with the term involving $S_{k+1}$. Choose the coordinate such that   $u_n(x_0)=|Du|(x_0)$ and $\{u_{ij}(x_0))\}_{1\le i,j \le n-1}=\{\tilde\lambda_i\delta_{ij}\}_{1\le i,j \le n-1}$ is diagonal, we have
\begin{align*}
	S_{k+1}=&u_{nn}S_{k}(\tilde\lambda)+S_{k+1}(\tilde\lambda)-\sum\limits_{i=1}^{n-1}S_{k-1}(\tilde\lambda|i)u_{ni}^2\\
	S_{k}=&u_{nn}S_{k-1}(\tilde\lambda)+S_{k}(\tilde\lambda)-\sum\limits_{i=1}^{n-1}S_{k-2}(\tilde\lambda|i)u_{ni}^2,
	\end{align*}
where $\tilde\lambda=(\tilde\lambda_1,\cdots,\tilde\lambda_{n-1})$ and  recall we use the notation $S_k=S_k(D^2 u)$. Then we get
\begin{align}\label{unn0724}
	S_{k+1}=&\frac{S_k(\widetilde\lambda)}{S_{k-1}(\widetilde\lambda)}S_k-\frac{S_k^2(\widetilde\lambda)}{S_{k-1}(\tilde\lambda)}+\sum_{i=1}^{n-1}u_{ni}^2\frac{S_{k}(\tilde \lambda|i)S_{k-2}(\tilde \lambda|i)-S_{k-1}^2(\tilde\lambda|i)}{S_{k-1}(\tilde \lambda)}+S_{k+1}(\tilde\lambda)\notag\\
	\le& \frac{S_k(\tilde\lambda)}{S_{k-1}(\widetilde\lambda)}S_k-\frac{S_k^2(\widetilde\lambda)}{S_{k-1}(\tilde\lambda)}+S_{k+1}(\tilde\lambda),
	\end{align}
where we use the  Newton's inequality (one can see the proof in \cite{chenkhessian}).\\
Inserting \eqref{unn0724} into \eqref{20724} and noting that $S_{m}(\tilde \lambda)=|Du|^{-2}S_{m+1}^{ij}u_iu_j=H_{m}|Du|^{m}$ is a global defined function, then we have
\begin{align}\label{20724}
	I_{a+a_0,b-1,k+1}(t_0)&-I_{a+a_0,b-1,k+1}(t)\notag\\ \le&(a+a_0)\int_{t}^{t_0}\int_{S_{s}}g^{a+a_0-1}g'|Du|^{b}H_k dAds\notag\\
	&+(k+1)\int_{t}^{t_0}\int_{S_{s}}g^{a+a_0}|Du|^{b-1}H_{k+1}dAds\notag\\
	&-b\int_{t}^{t_0}\int_{S_s}g^{a+a_0}|Du|^{b-1}\frac{H_k^2}{H_{k-1}}dAds
	+b\int_{t}^{t_0}\int_{S_s}g^{a+a_0}|Du|^{b-k-1}\frac{H_k}{H_{k-1}}S_kdAds.
\end{align}
Inserting \eqref{30724} and \eqref{20724} into \eqref{10724}, if $a=b-k+1\ge 0$, we obtain
\begin{align}
	J_{a+a_0,b,k}(t,t_0)\ge& -ba\int_{t}^{t_0}\int_{S_{s}}
	g^{a+a_0}|Du|^{b-k-1}\frac{H_k}{H_{k-1}}S_kdAds-(b+1)\int_{S_t}\Big(g^{a+a_0}|Du|^{b-k}S_k\Big)dA\notag\\
	&+a\int_{t}^{t_0}\int_{S_{s}}g^{a+a_0}|Du|^{b-1}H_{k-1}^{-1}\Big(c_{n,k}H_{k}^2-(k+1)H_{k-1}H_{k+1}\Big)dAds\notag\\
	&+a\int_{t}^{t_0}\int_{S_{s}}
	g^{a+a_0}
	|Du|^{b-1}H_{k-1} \mathcal{L}dAds,
	\end{align}
where the function $\mathcal{L}$ is defined by
\begin{align}\label{0723}
\mathcal{L}=&(b-c_{n,k})\Big(
\frac{H_k}{H_{k-1}}\Big)^2-(2a+a_0) (\log g)'|Du|\frac{H_k}{H_{k-1}}\notag\\
&+\Big((\log g)''+(a+a_0)((\log g)')^2\Big)|D u|^2.
\end{align}
Now we divide two cases to prove the $\mathcal{L}\ge 0$ under some restrictions on $a$ and $b$.

\textbf{Case1: $k<\frac{n}{2}$ and $\frac{n}{2}<k<n$.}\\
We choose $c_{n,k}=\frac{k(n-k-1)}{n-k}$.\\
Note that $\log g=\frac{n-k}{2k-n}\log(-u)$, then \begin{align}
	(\log g)''+(a+a_0)((\log g)')^2=&\frac{n-k}{n-2k}u^{-2}+(a+a_0)(\frac{n-k}{n-2k})^2u^{-2}\notag\\
=&(\frac{n-k}{n-2k})^2u^{-2}(\frac{n-2k}{n-k}+a+a_0)\notag\\
=&(b-c_{n,k})(\frac{n-k}{n-2k})^2u^{-2},
\end{align}
where we choose $a_0=-2\frac{n-2k}{n-k}$ and we use  $a=b-k+1$. We also have
\begin{align}
	-(2a+a_0)(\log g)'=2\frac{n-k}{n-2k}(b-c_{n,k})u^{-1}.
	\end{align}

 By direct manipulation, we have

\begin{align}
\mathcal{L}
	=&(b-c_{n,k})\Big(
	\frac{n-k}{n-2k}|D\log u|-\frac{H_k}{H_{k-1}}\Big)^2.
	\end{align}
Consequently,  we obtain
\begin{align}
	J_{a+a_0,b,k}(t,t_0)\ge&
	-ba\int_{t}^{t_0}\int_{S_{s}}
	g^{a+a_0}|Du|^{b-k-1}\frac{H_k}{H_{k-1}}S_kdAds-\int_{S_t}\Big(g^{a+a_0}|Du|^{b-k}S_k\Big)dA\notag\\
	&+a\int_{t}^{t_0}\int_{S_{s}}g^{a+a_0}|Du|^{b-1}H_{k-1}^{-1}\Big(c_{n,k}H_{k}^2-(k+1)H_{k-1}H_{k+1}\Big)dAds\notag\\
	&+a(b-c_{n,k})\int_{t}^{t_0}\int_{S_{s}}
	g^{a+a_0}
	|Du|^{b-1}\Big(
	\frac{n-k}{n-2k}|D\log u|-\frac{H_k}{H_{k-1}}\Big)^2
	\end{align}

\textbf{Case 2: $k=\frac{n}{2}$.}\\
We have $c_{n,k}=\frac{n}{2}-1>0$. We require $b\ge \frac{n}{2}-1$,  $a=b-\frac{n}{2}+1=b-c_{n,k}\ge 0$ and $a_0=0$.\\
Since $g=e^{u}$ and thus $(a+a_0)^{-1}(g^{a+a_0})''=(a+a_0)g^{a+a_0}$. We obtain
\begin{align}
	\mathcal{L}=a\Big(
	|D u|-\frac{H_k}{H_{k-1}}\Big)^2.
	\end{align}


%
	\end{proof}
From the above formula, we have the following almost monotonicity formula along the level set of $u^{\varepsilon}$ and we prove the first part of Theorem \ref{geometric0725}.

 \begin{lemma}\label{0727G1}
 	Let $u^{\varepsilon}$ be the solution of the approximating k-Hessian equation.
 	Assume $k<\frac{n}{2}$
 	and $b\ge \frac{k(n-k-1)}{n-k}$, then for any $t\in[-1,0)$, we have
 	\begin{align}
 		\frac{d}{dt}  I_{a,b,k}(t)\le C\varepsilon^2|t|^{\frac{2k}{n-2k}-1}.
 	\end{align}
Consequently, for any $-1\le t\le s<0$,
\begin{align}
	I_{a,b,k}(s)-	I_{a,b,k}(t)\le C\varepsilon^2.
	\end{align}
 	In particular, we have the following weighted inequality
 	\begin{align}\label{0723geometric}
 		\int_{\p\Omega}{|Du|^{b+1}H_{k-1}}\le \frac{n-2k}{n-k}\int_{\p\Omega}{|Du|^{b}H_{k}},
 	\end{align}
 where $u$ is the unique $C^{1,1}$ solution of the homogeneous $k$-Hessian equation  \eqref{case1Equa1.1}.
 \end{lemma}
\begin{remark}
	When $k=1$, \eqref{0723geometric} was proved by Agostiniani- Mazzieri \cite{AM2020CVPDE}.
	\end{remark}
\begin{proof}
By the Lemma \ref{0725mono1}, for any $-1\le t<t_0<0$, we have
\begin{align}
	-t^2  I'_{a,b,k}(t)&+t_0^2I'_{a,b,k}(t_0)\notag\\
	\ge&
	-ab\int_{\Omega_{t_0}\setminus\overline\Omega_t}(-u^{\varepsilon})^{a\frac{n-k}{2k-n}+2}|Du^{\varepsilon}|^{a-1}\frac{H_k}{H_{k-1}}S_k\notag\\
	&-(b+1)  \int_{S_{t}}(-u^{\varepsilon})^{a\frac{n-k}{2k-n}+2}|Du^{\varepsilon}|^{a-1}S_k.
	\end{align}
By the MacLaurin  inequality: $\frac{H_k}{H_{k-1}}\le \frac{C_{n-1}^{k}}{C_{n-1}^{k-1}}\Big(\frac{H_{k-1}}{C_{n-1}^{k-1}}\Big)^{\frac{1}{k-1}}$ and the uniform $C^2$-estimates of $u^\varepsilon$ (we also use  $|Du^{\varepsilon}|\ge c|x|^{1-\frac{n}{k}}$), for any  $x\in \Omega_t^{c}$, we have
\begin{align*}
	(-u^{\varepsilon})^{a\frac{n-k}{2k-n}+2}|Du^{\varepsilon}|^{a-1}\frac{H_k}{H_{k-1}}S_k\le& C(-u^{\varepsilon})^{a\frac{n-k}{2k-n}+2}|Du^{\varepsilon}|^{a-1}H_{k-1}^{\frac{1}{k-1}}|x|^{-n-2}\notag\\
	\le& C|x|^{a\frac{n-k}{k}+2\frac{2k-n}{k}}|x|^{(a-1)\frac{k-n}{k}}|x|^{-1}|x|^{-n-2}\notag\\
	=&C|x|^{-n-\frac{n}{k}},
	\end{align*}
then
 \begin{align*}
	\int_{\Omega_{t_0}\setminus\overline\Omega_t}(-u^{\varepsilon})^{a\frac{n-k}{2k-n}+2}|Du^{\varepsilon}|^{a-2}\frac{H_k}{H_{k-1}}S_k\le C\varepsilon^2|t|^{\frac{n}{n-2k}}.
\end{align*}
Similarly, we have
\begin{align*}
	\int_{S_{t}}(-u^{\varepsilon})^{a\frac{n-k}{2k-n}+2}|Du^{\varepsilon}|^{a-1}S_k\le C\varepsilon^2|t|^{\frac{n}{n-2k}},
	\end{align*}
where we use $|S_t|\le C|t|^{^{-\frac{k(n-1)}{n-2k}}}$ (see Lemma \ref{st0723}).

Thus we get
\begin{align}\label{110725}
	-t^2I'_{a,b,k}(t)&+t_0^2 I'_{a,b,k}(t_0)
\ge-C\varepsilon^2|t|^{\frac{n}{n-2k}}.
	\end{align}
 By the uniform $C^{2}$ estimates for $u^{\varepsilon}$ and  $|Du^{\varepsilon}|\ge c|x|^{1-\frac{n}{k}}$,
we have for any $t_0\in[-1,0)$
\begin{align}
	t_0^2 |I'_{a,b,k}(t_0)|\Big|\le C|t_0|.
	\end{align}

Let $t_0$ tend to $0$ in \eqref{110725}, we have
\begin{align}
	  I'_{a,b,k}(t)\le C\varepsilon^2|t|^{\frac{2k}{n-2k}-1}.
	\end{align}
In particular,  taking $t=-1$, we have
\begin{align}
 	 I'_{a,b,k}(-1)\le C\varepsilon^2.
	\end{align}
On the other hand, by \eqref{0726I}, we have
\begin{align}
	 I_{a,b,k}(-1)
\ge a\frac{n-k}{n-2k}\int_{\p\Omega}{|Du^{\varepsilon}|^{b+1}H_{k-1}}- a\int_{\p\Omega}{|Du^{\varepsilon}|^{b}H_{k}}.
\end{align}
Consequently, we get
\begin{align}\label{0723final}
\frac{n-k}{n-2k}\int_{\p\Omega}{|Du^\varepsilon|^{b+1}H_{k-1}}- \int_{\p\Omega}{|Du^\varepsilon|^{b}H_{k}}\le C\varepsilon^2
\end{align}
Since $|Du^{\varepsilon}|$ converges to $|Du|$ on $\p \Omega$, we finish the proof of \eqref{0723geometric} by taking $\varepsilon\rightarrow 0$ in \eqref{0723final}.
	\end{proof}

Next we prove the second part  of Theorem \ref{geometric0725}.
\begin{lemma}\label{0727G2}
	Assume $k=\frac{n}{2}$ and $b>\frac{n}{2}-1$. We have
	\begin{align}
	I'_{a, b, k}(t)\le C\varepsilon^2 e^{-2t},
	\end{align}
In particular, we have
\begin{align}
	\int_{\p\Omega}|Du|^{b+1}H_{k-1}
	\le \int_{\p\Omega}|Du|^{b}H_{k},
\end{align}
where $u$ is the unique $C^{1,1}$ solution the homogeneous $k$-Hessian equation  \eqref{case3Equa1.3}.
	\end{lemma}
\begin{proof}
By Lemma \ref{0725mono1} and similar as the proof in the above lemma, for any $0\le t<t_0$, we have
\begin{align*}
I'_{a,b,k}(t_0)\ge I'_{a,b,k}(t)- C\varepsilon^2e^{-2t}.
\end{align*}
 By integrating the above form $t$ to $t_0$, we have
\begin{align*}
	I_{a,b,k}(t_0)-I_{a,b,k}(t)\ge
	(I'_{a,b,k}(t)- C\varepsilon^2e^{-2t}) (t_0-t),
	\end{align*}
Since $I_{a,b,k}(t)$ is uniformly bounded which follows from the $C^2$-estimates of $u^{\varepsilon}$ and $|Du^{\varepsilon}|\ge c |x|^{1-\frac{n}{k}}$, we have
\begin{align*}
	\Big(I'_{a,b,k}(t)- C\varepsilon^2e^{-2t}\Big)(1-tt_0^{-1})\le t_0^{-1}(I_{a,b,k}(t_0)-I_{a,b,k}(t))\le Ct_0^{-1}.
	\end{align*}
Let $t_0$ tend to $0$,  we obtain
\begin{align*}
	I'_{a,b,k}(t)\le C\varepsilon^2e^{-2t}.
	\end{align*}
On the other hand, we have\begin{align*}
	I'_{a,b,k}(0)\ge& a\int_{\p\Omega}|Du^{\varepsilon}|^{b+1}H_{k-1}
	-a\int_{\p\Omega}|Du^{\varepsilon}|^{b}H_{k}.
\end{align*}
Combining the above two inequalities and noting that $|Du^{\varepsilon}|\rightarrow |Du|$,
we get
\begin{align}
	\int_{\p\Omega}|Du|^{b+1}H_{k-1}
	\le \int_{\p\Omega}|Du|^{b}H_{k}
	\end{align}
	\end{proof}

When $\frac{n}{2}<k<n$, we have the following inequality.

\begin{lemma}
	Let $u^{\varepsilon}$ be the solution of the approximating k-Hessian equation.
	Assume $k>\frac{n}{2}$,
	and $b\ge -k+1$, then for any $1\le t\le t_0<\infty$, we have
	\begin{align}
		t^2I'_{a,b,k}(t)  -t_0^2I'_{a,b,k}(t_0)\le C\varepsilon^2|t|^{\frac{2k}{n-2k}-1}.
	\end{align}
	
\end{lemma}

{\bf Acknowledgements:}
The first author was supported by  National Natural Science Foundation of China (grants 11721101 and 12141105) and National Key Research and Development Project (grants SQ2020YFA070080). The second author is supported by NSFC grant  No. 11901102.

\bibliographystyle{plain}
\bibliography{Exterior-Homogeneous-k-Hessian-Ma-Zhang}

\end{document}